\documentclass[a4paper,12pt]{article}

\usepackage{bm} 
\usepackage{amsmath}
\usepackage{amssymb}
\usepackage{amsfonts}
\usepackage{dsfont}
\usepackage{array}
\usepackage{mathabx}
\usepackage{amsthm}
\usepackage{enumerate}
\usepackage{upgreek}
\usepackage{arydshln}
\usepackage{graphicx}
\usepackage{epstopdf}
\usepackage{commath}
\usepackage{bbold}
\usepackage{latexsym}
\usepackage{hhline}
\usepackage{caption}
\usepackage{color}

\newtheorem{theorem}{Theorem}[section]
\newtheorem{proposition}[theorem]{Proposition}
\newtheorem{lemma}[theorem]{Lemma}
\newtheorem{corollary}[theorem]{Corollary}

\theoremstyle{definition}

\usepackage{geometry}
\makeatletter
\renewcommand*\env@matrix[1][\arraystretch]{%
  \edef\arraystretch{#1}%
  \hskip -\arraycolsep
  \let\@ifnextchar\new@ifnextchar
  \array{*\c@MaxMatrixCols c}}
\makeatother

\numberwithin{equation}{section}

\usepackage{color}
\definecolor{darkred}{rgb}{0.8,0,0}

\begin{document}
\pagenumbering{arabic}


\newcommand{\R}{\mathbb{R}}
\newcommand{\stochle}{\le_{\mathrm{st}}}
\newcommand{\hazardle}{\le_{\mathrm{hr}}}
\newcommand{\N}{\mathbb{N}}
\newcommand{\Z}{\mathbb{Z}}
\newcommand{\C}{\mathbb{C}}
\newcommand{\ind}[1]{\mathds{1}_{#1}}
\newcommand{\esp}[1]{\mathbb{E}\left[ #1 \right]}
\newcommand{\var}[1]{\mathbb{V}ar\left[ #1 \right]}
\newcommand{\prob}[1]{\mathbb{P}\left( #1 \right)}
\newcommand{\cqfd}{\hfill $\Box$}
\newcommand{\eps}{\varepsilon}
\newcommand{\red}[1]{\textcolor{red}{#1}}
\newcommand{\blue}[1]{\textcolor{blue}{#1}}
\newcommand{\green}[1]{\textcolor{olive}{#1}}


\newcommand{\Sr}{\mathcal{S}}
\newcommand{\Er}{\mathcal{E}}

\newcommand{\Psitet}{\Psi_{\theta}}
\newcommand{\Phitet}{\Phi_{\theta}}
\newcommand{\Utet}{U_{\theta}}


\begin{center}
{\Large \bf Ruin problems for risk processes with dependent phase-type claims}

\vspace{6mm}
{{\bf Oscar Peralta}\footnote {Cornell University, School of Operations Research and Information Engineering,
Rhodes Hall, 
Ithaca 14850, New York, 
United States. E-mail: op65@cornell.edu}, {\bf Matthieu Simon}
\footnote {D\'epartement de Math\'ematique, Place du Parc 20 , B-7000 Mons, Belgium. {\it E-mail}: matthieu.simon@umons.ac.be. Previously at Universitat de Barcelona, Departament de Matem\`atica Econ\`omica, Financera i Actuarial, 690 Avinguda Diagonal, Barcelona E-08034, Spain.}}
\end{center}

\vspace{1mm}

\begin{abstract}
We consider continuous time risk processes in which the claim sizes are dependent and non-identically distributed phase-type distributions. 
The class of distributions we propose is easy to characterize and allows to incorporate the dependence between claims in a simple and intuitive way. It is also designed to facilitate the study of the risk processes by using a Markov-modulated fluid embedding technique. Using this technique, we obtain simple recursive procedures to determine the joint distribution of the time of ruin, the deficit at ruin and the number of claims before the ruin. We also obtain some bounds for the ultimate ruin probability. Finally, we provide a few examples of multivariate phase-type distributions and use them for numerical illustration.
\end{abstract}

\vspace{3mm}

\vspace{3mm}


\section{Introduction} \label{A_intro}

A risk process is a stochastic processes $\{R(t)\,|\, t \ge 0\}$ of the form
\begin{equation}\label{Eq_reserve}
R(t)= u + ct - \sum_{k=1}^{N(t)}Y_k,
\end{equation}
where $u\ge 0$, $c>0$, $\{N(t) \,|\,t\ge 0\}$ is a counting process and $\{Y_k\}_{k\ge 1}$ is a family of nonnegative random variables. Such processes are commonly used in risk theory to represent the level of reserves of an insurance company that collects premiums at continuous rate $c$ and reimburses the claims $Y_k$ that occur according to $\{N(t)\}$. A fundamental problem in this context is to determine the  probability that the reserves becomes negative in finite time, and, when it happens, after how long.

The most classical risk process was introduced by Cram\'er \cite{Cramer} and is called the Cram\'er-Lundberg process. It assumes that the variables $\{Y_k\}_{k\ge 1}$ are i.i.d. and that $\{N(t)\}$ is a Poisson process, independent of the claim sizes. Since then, this model was generalised in various ways: we refer to Asmussen and Albrecher \cite{AsmussenRuin1} for an introduction to the fundamental risk model and its major extensions.

Most of the results in the literature concern risk processes in which the claim sizes are independent and identically distributed. However, there are many circumstances where models with dependent claims appear to be more appropriate. Indeed, large groups of people are often subject to common risks stemming from economic, environmental or epidemiological factors, for instance. Moreover, claims that are not identically distributed can be relevant in various situations, for example when the reaction capacity of the threatened population can evolve over time.
Different risk models with dependent claims have already been considered in the literature; see e.g.\ Albrecher \textit{et al.} \cite{Albrecher1}, Constantinescu \textit{et al.} \cite{Constantinescu}, Bladt \textit{et al.} \cite{Oscar} and references therein. 

In this paper, we consider a risk process of the form \eqref{Eq_reserve} in which the claim sizes are random vectors taken from a family of multivariate phase-type distributions. The random vectors that we consider have several advantages: 
firstly, their components are neither independent nor identically distributed in general, and dependence between components can be incorporated in a simple and intuitive way.
Secondly, they are easy to characterise and admit simple and explicit expressions for the joint density and  the correlation matrix.
Finally, they are designed to facilitate the analysis of the corresponding risk process through the study of an appropriate embedded Markov-modulated fluid flow.

We use this embedding method to derive an explicit formula for a transform of the ruin time $T$, the severity at ruin $-R(T)$ and the number of claims until ruin $N(T)$ in terms of some first passage matrices related to the embedded fluid flow. We present a recursive procedure to compute this transform numerically. We also briefly explain how the approach can be easily extended to the analysis of risk processes in a random environment. Next, we obtain simple bounds for the ultimate ruin probability of our model, in terms of the ruin probability for risk processes with independent Erlang claims. These bounds are quite improvable, but they allow to check whether the ultimate ruin is almost sure or not in a variety of situations.
Finally, we conclude with some numerical illustrations where we compare the ruin probabilities for different risk models with multivariate phase-type distributed claims.

The paper is organised as follows. 
In Section \ref{A_PH}, we first briefly review the main properties of univariate phase-type distributions. Next, we introduce our class of multivariate phase-type distribution and provide a few examples.
In Section \ref{A_continuous}, we turn to the analysis of the risk process. We first detail the construction of the embedded Markov-modulated fluid flow and use it to determine a transform of  $T$,  $-R(T)$ and  $N(T)$. We then present our bounds for the ultimate ruin probability.
Finally, in Section \ref{AAA_illustrations}, we present a few numerical illustrations.

\section{Dependent phase-type distributions} \label{A_PH}

\subsection{Univariate phase-type distributions}\label{subsec:uniPH}

Throughout the text, $\bm{0}$ denotes a vector of zeros and $\bm{1}$ denotes a vector of ones, with appropriate size and orientation. 

Let $\{\varphi(t)\,|\,t\ge 0\}$ be a time-homogeneous Markov jump process defined on a state space $\{\star\}\cup \Sr$, where $\Sr$ contains $p<\infty$ transient states and $\star$ is an absorbing state. The generator of $\varphi$ is then of the form
\begin{equation*}
\Lambda =
\begin{bmatrix}
\, 0 \!\!\! & \vline & & \bm{0} & & \\ 
\hline
& \vline & & & & \\ 
\, \bm{a} \!\!\! & \vline & & A & & \\ 
& \vline & & & &
\end{bmatrix},
\end{equation*}
where $A$ is a  $p \times p$ matrix containing the transition rates between the transient states and $\bm{a}=-A\bm{1}$ is the vector containing the transition rates from the transient states to the absorbing state. The initial probability vector $\bm{\alpha}$ of $\{\varphi(t)\}$ on $\Sr$,  with components $\alpha_i = \prob{\varphi(0) = i}$ for $i \in \Sr$, is assumed to satisfy $\bm{\alpha}\bm{1} = 1$.
We say that a random variable $Y$ has a phase-type distribution of size $p$ with initial distribution $\bm{\alpha}$ and sub-generator matrix $A$, and write $Y\sim \mathrm{PH}(\bm{\alpha},A)$, if $Y$ is distributed as the time before absorption in  $\{\varphi(t)\}$:
\begin{equation*}
Y \stackrel{d}{=} \inf\{t \ge 0 \,|\, \varphi(t) = \star\}.
\end{equation*}

Phase-type distributions have been popularized by the work of Neuts \cite{NeutsPH, Neuts}. Since then, they have been used in many application fields. One of their advantages is that they characterisation is easy and intuitive: if $Y \sim \mathrm{PH}(\bm{\alpha},A)$, then its density function is
\begin{equation}
f_Y(t) =  \bm{\alpha} e^{At} \bm{a}   \label{densitePHdistr}
\end{equation}
for $t \ge 0$, and the moments of $Y$ are given by
\begin{equation}
\esp{Y^k} = k! \, \bm{\alpha} (-A)^{-k} \bm{1}, ~~\quad k \in \N, \label{PHexpect}
\end{equation}
see e.g.\ Latouche and Ramaswami \cite[Chapter 2]{MAM} or Neuts \cite[Chapter 2]{Neuts}.
Note that the expression for the density is intuitive since the component $(e^{At})_{ij}$ is the probability that $\varphi$ is in state $j$ (and therefore not absorbed yet) at time $t$ if it started from state $i$, while the vector $\bm{a} \, dt$ contains the probabilities of absorption on the infinitesimal time interval $[t \,, \, t+dt]$ (conditional on the process not being absorbed by time $t$).  Also, the formula $\esp{Y} = \bm{\alpha} (-A)^{-1} \bm{1}$ is straightforward since $(-A^{-1})_{ij}$ is the average time spent in state $j \in \Sr$ before absorption given that the starting state is $i \in \Sr$.

Phase-type distributions benefit from various interesting closure properties. For instance, they are stable under convolutions: 
if $Y_1 \sim \mathrm{PH}(\bm{\alpha}_1,A_1)$ and $Y_2 \sim \mathrm{PH}(\bm{\alpha}_2,A_2)$ are independent, then $Y_1 + Y_2 \sim \mathrm{PH}(\bm{\beta}_1,B_1)$ where
\begin{equation*}
B_1 = \begin{bmatrix}[1.2]
A_1 & \vline & \bm{a}_1 \bm{\alpha}_2 \\
\hline
0 & \vline & A_2
\end{bmatrix}\quad\mathrm{with}\quad\bm{a}_1=-A_1\bm{1},
\end{equation*}
and
$\bm{\beta}_1 = \left[\bm{\alpha}_1 ~\, \bm{0}\right]$ (see \cite[Section 2.6]{MAM}). 
Moreover, if $\delta$ is a Bernoulli random variable independent from $Y_1$ and $Y_2$ with parameter $p$, then $\delta Y_1+(1-\delta)Y_2\sim\mathrm{PH}(\bm{\beta}_2,B_2)$ where
\begin{equation*}
B_2 = \begin{bmatrix}[1.2]
A_1 & \vline & 0 \\
\hline
0 & \vline & A_2
\end{bmatrix}\quad\mathrm{and}\quad\bm{\beta}_2 = \left[p\bm{\alpha}_1 ~~ (1-p)\bm{\alpha}_2\right].
\end{equation*} 
 
Finally, let us mention that the class of phase-type distributions is dense (in the weak convergence sense) within the class of distributions with support on $[0,\infty)$ (see e.g.\ Breuer and Baum \cite[Theorem 9.14]{BookBreuer}). Together with the previous properties, this makes phase-type distributions a very powerful tool to model and analyse a wide range of random phenomena.

\subsection{Multivariate phase-type distributions} \label{BBB_mult_PH}

Different classes of multivariate phase-type distributions have been considered in the literature. The first one was introduced by Assaf \textit{et al.}\ \cite{Assaf} who proposed the following definition: let $\varphi$ be a Markov jump process as in Section \ref{subsec:uniPH} and consider a collection of subsets $\{\mathcal{S}_i\}_{1\le i\le n}$ such that $\mathcal{S}_i\subset\{\star\}\cup\mathcal{S}$, $\mathcal{S}_i$ is stochastically closed for all $i$ (i.e.\ if the process enters $\mathcal{S}_i$, it never leaves it) and $\bigcap^n_{i=0}\mathcal{S}_i = \{\star\}$. A random vector $\bm{Y}^{(n)} = (Y_1,Y_2,\dots,Y_n)$ is said to follow a multivariate phase-type ($\mathrm{MPH}$) distribution if 
\[(Y_1, Y_2,\dots, Y_n )\stackrel{d}{=}(\nu_1,\nu_2,\dots,\nu_n)\quad\mathrm{where}\quad\nu_i=\inf\{t \ge 0 \,|\, \varphi(t) \in \mathcal{S}_i\}.\]
The authors derived various properties of these distributions, including a closed expression for their joint density. 

The class of ($\mathrm{MPH}$) distribution was then extended to the family of $\mathrm{MPH}^*$ distributions by Kulkarni \cite{Kulkarni}. A random vector $\bm{Y}^{(n)}$ is said to follow a $\mathrm{MPH}^*$ distribution if there exists a collection of non-negative numbers $\{r(h,i): h\in\mathcal{S}, 1\le i\le n\}$ such that 
\[(Y_1, Y_2,\dots, Y_n )\stackrel{d}{=}(\nu_1^*,\nu_2^*,\dots,\nu_n^*)\quad\mathrm{where}\quad\nu_i^*=\int_0^\infty r(\varphi(s),i)\,ds\quad\mathrm{and}\quad r(\star, i)=0.\] Unlike the $\mathrm{MPH}$ distributions, there is no closed-form formula for the joint density of $\bm{Y}^{(n)}\sim\mathrm{MPH}^*$, so its analysis is, in most cases, limited to its Laplace transform. This class was further extended in Bladt and Nielsen \cite{bladt2010multivariate} who considered those random vectors $\bm{Y}^{(n)}$ such that $\sum_{i=1}^n c_iY_i$ follows an univariate phase-type distribution for any $c_i\ge 0$.

\vspace{1em}

In this paper, we introduce another class of multivariate phase-type distributions which is  suitable for our analysis of risk processes. For such processes, the claim sizes are determined sequentially over time, that is, they are sampled one after the other. This motivates the following definition.

Let $n \in \N$ and $\{\varphi(t)\,|\,t\ge 0\}$ be a Markov jump process defined on the state space 
$$\{\star\} \cup \Sr_1 \cup \Sr_2 \cup \cdots \cup \Sr_n,$$
 where each $\Sr_i$ is some finite subset of transient states and $\star$ is an absorbing state. Assume that its generator is of the following form when written according to the state decomposition above:
\begin{equation}\label{gen_G}
\Lambda = 
\begin{bmatrix}
 0  & \vline & & && \bm{0} & & \\ 
\hline
\bm{0} & \vline & \! \! A_1 & D_1 & 0 & 0 & \cdots & 0 & 0 \\
\bm{0} & \vline & 0& A_2 & D_2 & 0 & \cdots & 0 & 0 \\
\bm{0} & \vline & 0&0& A_3 & D_3 & \cdots & 0 & 0 \\
\vdots & \vline & \vdots &\vdots &\vdots & \vdots  & \ddots & \vdots  & \vdots \\
\bm{0} & \vline & 0&0&0& 0 & \cdots & A_{n-1} & D_{n-1} \\
D_n\bm{1} & \vline & 0&0&0& 0 & \cdots & 0 & A_n 
\end{bmatrix},
\end{equation} 
and the initial state is given by the probability vector $\bm{\pi} = [0\,|\,\bm{\alpha}~\bm{0}~\cdots~\bm{0}]$, so that the process starts from the subset $\Sr_1$. Here, $(A_k)_{ij}$ is the transition rate from $i\in \Sr_k$ to $j\in \Sr_k$, $j \neq i$. For $k<n$, the matrix $D_k$ contain the transition rates from $\Sr_k$ to $\Sr_{k+1}$. Finally, the vector $D_n\bm{1}$ contains the transitions from $\Sr_n$ to the absorbing state $\star$.

From now on, we say that the vector $\bm{Y}^{(n)} = (Y_1,Y_2,...,Y_n)$ follows a multivariate phase-type distribution if $Y_k$ is the amount of time spent by $\{\varphi(t)\}$ in the subspace $\Sr_k$ before absorption:
\begin{equation}\label{eq:defYk1}
Y_k=\int_0^\infty \ind{\varphi(s)\in \Sr_k}\,ds.
\end{equation}
Note that each component $Y_k$ is almost surely finite since each subset $\Sr_k$ is assumed to be transient.
Our class of multivariate phase-type distributions is clearly a subset of the class $\mathrm{MPH}^*$ introduced by Kulkarni \cite{Kulkarni}. The case $n=2$ was analysed in Bladt \textit{et al.} \cite[Theorem 6.10]{Oscar}, where the authors showed that the vector $\bm{\alpha}$ and the matrices $A_1$, $D_1$, $A_2$ and $D_2$ can be chosen in such a way that $Y_1$ and $Y_2$ are phase-type-distributed with any feasible goal covariance.

From the structure \eqref{gen_G} of the generator $\Lambda$, it is clear that the components $Y_k$ are determined sequentially in $\varphi$: the process starts in $\Sr_1$ and $Y_1$ is known as soon as it leaves $\Sr_1$ for $\Sr_2$. Then, $Y_2$ is known as soon as the process leaves $\Sr_2$ for $\Sr_3$, and so on.
This is the key feature that will allow us to apply the fluid embedding technique in the next section, where we consider risk processes with multivariate phase-type claims.

\vspace{1em}

Let us first have a look at the distribution of $\bm{Y}^{(n)}$. From (\ref{eq:defYk1}) and the structure \eqref{gen_G} of $\Lambda$, it is easy to see that $$Y_k \sim \mathrm{PH}(\bm{\gamma}_k,A_k)$$ 
where $\bm{\gamma}_1=\bm{\alpha}$ and for $k\ge 2$,
\begin{equation} \label{gammak}
\bm{\gamma}_k = \bm{\alpha}(-A_1)^{-1}D_1(-A_2)^{-1}D_2\cdots (-A_{k-1})^{-1}D_{k-1}.
\end{equation}
This follows from the fact that $[(-A_\ell^{-1})D_\ell]_{ij}$ is the probability that the process $\varphi$ is in phase $j$ when entering $\Sr_{\ell+1}$ given $\varphi$ entered $\Sr_{\ell}$ in phase $i$. So, the marginal density and the moments of $Y_k$ are obtained from \eqref{densitePHdistr} and \eqref{PHexpect}. The components of $\bm{Y}^{(n)}$ can be dependent since the state occupied by $\varphi$ when entering a subset $\Sr_k$ depends on the trajectories of $ \varphi$ in $\Sr_1 \cup \cdots \cup \Sr_{k-1}$. The joint distribution of $\bm{Y}^{(n)} = (Y_1,Y_2,...,Y_n)$ is given below, {and is a straightforward extension of the corresponding results for absorbing Markov arrival processes (see e.g.\ Latouche and Ramaswami \cite{MAM}).}
\begin{proposition}\label{prop:density1}
The density function of $\bm{Y}^{(n)}$ is given by
\begin{equation}
f(y_1,y_2,...,y_n) = \bm{\alpha}e^{A_1y_1}D_1e^{A_2y_2}D_2\cdots e^{A_ny_n}D_n \bm{1}. \label{joint_density}
\end{equation}
for $y_i \ge 0, ~i=1,2,...,n$.
\end{proposition}
\begin{proof}
The density is easily obtained by induction on $n$: assuming the form \eqref{joint_density} for $n-1$ components and denoting by $\tau_n$ the first time $\varphi$ is in $\Sr_n$,
\begin{align*}
&\prob{Y_1\le y_1,...,Y_n\le y_n} \\
&~~= \sum_{j \in \Sr_n} \int_0^{y_1} \cdots \int_0^{y_{n-1}} \prob{Y_1\le y_1,...,Y_n\le y_n \,|\, Y_1 = x_1,...Y_{n-1}=x_{n-1}, \varphi(\tau_n)=j} \\
&\hspace{21em} .\, d\prob{Y_1 \le x_1,...Y_{n-1} \le x_{n-1}, \varphi(\tau_n)=j} \\
&~~= \sum_{j \in \Sr_n} \int_0^{y_1} \cdots \int_0^{y_{n-1}} \prob{Y_n\le y_n \,|\, \varphi(\tau_n)=j} \left(\bm{\alpha}e^{A_1x_1}D_1 \cdots e^{A_{n-1}x_{n-1}}D_{n-1} \right)_j \, dx_1 \cdots dx_{n-1} \\
&~~=\int_0^{y_1} \cdots \int_0^{y_{n}} \bm{\alpha}e^{A_1x_1}D_1 \cdots e^{A_{n-1}x_{n-1}}D_{n-1} e^{A_{n}x_{n}}D_{n} \bm{1} \, dx_1 \cdots dx_{n}.
\end{align*}
Differentiating with respect to $y_1,\dots, y_n$ in the last equality yields \eqref{joint_density}.
\end{proof}

A closed expression for the covariances between the components of $\bm{Y}^{(n)}$ can also be easily obtained:
\begin{proposition} \label{prop_cov}
For $1\le k < \ell \le n$,
\begin{equation}\label{cov}
\begin{aligned}
Cov(Y_k,Y_\ell) &= \bm{\gamma}_k (-A_k)^{-2}D_k (-A_{k+1})^{-1}D_{k+1}\cdots (-A_{\ell-1})^{-1}D_{\ell-1} (-A_{\ell})^{-1}\bm{1}\\
&~~- \left(\bm{\gamma}_k (-A_{k})^{-1} \bm{1}\right)\,\left(\bm{\gamma}_\ell (-A_{\ell})^{-1} \bm{1}\right),
\end{aligned}
\end{equation}
where the vectors $\bm{\gamma}_k$ are given in \eqref{gammak}.
\end{proposition} 
\begin{proof}
The covariance between $Y_k$ and $Y_\ell$ is given by $Cov(Y_k,Y_\ell) = \esp{Y_kY_\ell} - \esp{Y_k} \esp{Y_\ell}$,
and $\esp{Y_k}=\bm{\gamma}_k (-A_{k})^{-1} \bm{1}$, $\esp{Y_\ell}=\bm{\gamma}_l (-A_{\ell})^{-1} \bm{1}$ are known from \eqref{PHexpect}. To obtain $\esp{Y_kY_\ell}$, we start from the Laplace transform of $\bm{Y}^{(n)}$, which is easily derived by integration in \eqref{joint_density}. For a $n$-dimensional vector $\bm{\theta}$ of nonnegative components,
\begin{equation}
\esp{e^{-\langle \bm{\theta},\bm{Y}^{(n)}\rangle}} = \bm{\alpha}(\theta_1I-A_1)^{-1}D_1(\theta_2I - A_2)^{-1}D_2\cdots (\theta_nI-A_{n})^{-1}D_{n}\bm{1}. \label{Laplace_tf}
\end{equation}
Using that 
$$
\frac{d}{d\theta}(\theta I-M)^{-1} = - (\theta I-M)^{-2}
$$
for any square matrix $M$ such that $\theta I-M$ is invertible, we obtain
\begin{align*}
\esp{Y_kY_\ell} &= \left((d/d\theta_k)(d/d\theta_\ell) \esp{e^{-\langle \bm{\theta},\bm{Y}^{(n)}\rangle}}\right)\bigg|_{\bm{\theta}=\bm{0}} \\
&= \bm{\gamma}_k (-A_k)^{-2}D_k (-A_{k+1})^{-1}D_{k+1}\cdots (-A_{\ell-1})^{-1}D_{\ell-1} (-A_{\ell})^{-2} D_\ell (-A_{\ell+1})^{-1}D_{\ell+1} \\
& \hspace{5em} \cdots (-A_{n-1})^{-1}D_{n-1}(-A_{n})^{-1}D_{n}\bm{1}.
\end{align*}
It suffices to use that $D_i \bm{1} = -A_i\bm{1}$ for all $\ell\le i\le n$ to obtain the announced expression from the last equality.
\end{proof}

We now present some examples of vectors $\bm{Y}^{(n)}$ following a multivariate phase-type distribution. They will be used later for illustration in the setting of risk processes. 

\vspace{1em}

\noindent {\bf Example 1.} Let $\bm{Y}^{(n)}$ be a vector of $n$ independent  random variables $Y_k \sim \mathrm{PH}(\bm{\alpha}_k,A_k)$. Then $\bm{Y}^{(n)}$ has a multivariate phase-type distribution of representation \eqref{gen_G} where $\bm{\alpha}  = \bm{\alpha}_1$, $D_k = (-A_k \bm{1}) \bm{\alpha}_{k+1}$ for $1 \le k\le n-1$ and $D_n \bm{1} = -A_n \bm{1}$.
\vspace{1em}

\noindent {\bf Example 2.} Let $\{U_1,...U_n\}$ and $\{V_1,...V_n\}$ be two collections of independent random variables such that $U_k \sim \mathrm{PH}(\bm{\beta},B)$ and $V_k \sim \mathrm{PH}(\bm{\gamma},G)$ for all $k \in \{1,...,n\}$. Let $r, r_k, p_k$ ($2\le k\le n$) be some constants in $[0,1]$. Define the random vector $\bm{Y}^{(n)}$ as follows: First, $Y_1 = U_1$ with probability $r$ and $Y_1 = V_1$ with the complementary probability $1-r$.
Next, the value of $Y_k$, $k=2,3,...,n$ is chosen sequentially according to the value of $Y_{k-1}$: if $Y_{k-1}=U_{k-1}$, then
$$
Y_k = \left\{
\begin{array}{ll}
U_k &\textrm{with probability } r_k, \\
V_k &\textrm{with probability } 1-r_k.
\end{array}
\right.
$$
If $Y_{k-1}=V_{k-1}$, then
$$
Y_k = \left\{
\begin{array}{ll}
U_k &\textrm{with probability } 1 - p_k, \\
V_k &\textrm{with probability } p_k.
\end{array}
\right.
$$
The vector $\bm{Y}^{(n)}$ has a multivariate phase-type distribution with parameters
\[A_k \equiv A = \begin{bmatrix}[1.2]
B & 0 \\
0  & G
\end{bmatrix}, \quad ~ D_k=\begin{bmatrix}[1.2]
r_k \bm{b}\bm{\beta} & (1-r_k)\bm{b}\bm{\gamma} \\
(1-p_k) \bm{g}\bm{\beta} & p_k \bm{g}\bm{\gamma}
\end{bmatrix},  \quad ~ \bm{\alpha}=[r\bm{\beta} ~\, (1-r)\bm{\gamma}],\]
where $\bm{b}=-B\bm{1}$ and $\bm{g}=-G\bm{1}$.

\vspace{1em}

\noindent {\bf Example 3.} Fix $m>1$ and consider the random vector $\bm{Y}^{(n)}$ of  representation \eqref{gen_G} with the $m \times m$ matrices $A_k$ and $D_k$ such that $\forall k \ge 1$,
\begin{equation*}
A_k = \begin{bmatrix}
-\mu_k & \mu_k p_k & 0 & \cdots & 0 \\
0 & -\mu_k & \mu_k p_k  & \cdots & 0 \\
0 & 0 & -\mu_k  & \cdots & 0 \\
\vdots & \vdots & \vdots & \ddots & \vdots \\
0 & 0& 0  & \cdots & -\mu_k \\
\end{bmatrix}, ~~
D_k =
\begin{bmatrix}
& & & & \vline & \\
& & \mu_k (1-p_k) P & & \vline & \!\! \bm{0} \\
& & & & \vline &  \\
\hline
& & \mu_k \bm{\beta}_k & & \vline & \!\!  0
\end{bmatrix},
\end{equation*}
where $P$ is an $(m-1) \times (m-1)$ stochastic matrix, $\bm{\beta}_k$ is a  probability vector with $m-1$ components, $\mu_k$ is a positive rate and  $0 < p_k <1$. The initial probability vector $\bm{\alpha}$ is arbitrary.
Here, the states in $A_k$ can be seen as $m$ successive stages $1,2,...,m$ of duration Exp($\mu_k$) each. If the process $\varphi$ enters $\Sr_k$ in the ${\ell}$-th stage, then $Y_k$ is the sum of at most $m-\ell+1$ i.i.d variables Exp($\mu_k$).   The dependences between the components of $\bm{Y}^{(n)}$ comes from the fact that the initial stage visited in $\Sr_{k+1}$ depends on the last stage visited by $Y_k$ through the transition matrix $P$ and the vector $\bm{\beta}_k$.


\section{Risk process with multivariate phase-type claims} \label{A_continuous}

In this section, we consider the risk process $\{R(t)\,|\, t\ge 0\}$ given by
\begin{equation}\label{risk_proc_MPH}
R(t)= u + ct - \sum_{k=1}^{N(t)}Y_k,
\end{equation}
where $\{N(t)\,|\, t\ge 0\}$ is a Poisson process of intensity $\lambda>0$ and, for all $n \ge 1$, the vector $(Y_1, Y_2, \dots, Y_n)$ is independent of $\{N(t)\}$ and follows a multivariate phase-type distribution with representation (\ref{gen_G}). The parameters $u,c >0$ correspond respectively to the initial level of reserves and the premium rate. We are interested in the distribution of three statistics related to this model: the time of ruin 
\begin{equation}
T = \inf\{t \ge 0 \,|\, R(t)<0\},  \label{def_Truine}
\end{equation}
the deficit at ruin $-R(T)$ and the number $N(T)$ of claims that occurred up to time $T$. Our aim is to determine their joint distribution through the transform
\begin{equation} \label{def_tf_ruine}
\esp{e^{-\theta T}\ind{T<\infty,\, N(T)\le s,\, -R(T)\ge y}\,|\, R(0)=u},
\end{equation}
for $\theta \ge 0$, $s \in \N_0$ and $y \ge 0$.

We are going to determine \eqref{def_tf_ruine} through the study of a Markov-modulated fluid flow closely related to the risk process \eqref{risk_proc_MPH}. This method is sometimes called fluid embedding and has already been used in the setting of risk theory (see e.g.\ Badescu and Landriault \cite{Badescu} for an overview). Compared to existing models, a difference here is that we need to consider an infinite phase space for the embedded fluid flow. The construction of the fluid model and its links with the risk process are detailed in Section \ref{BBB_MMFF}. In Section \ref{BBB_tf}, we use this fluid approach to derive an expression for \eqref{def_tf_ruine}.

\subsection{The embedded Markov-modulated fluid flow} \label{BBB_MMFF}

A Markov-modulated fluid flow (MMFF) is a stochastic process $\{(X(t), \phi(t)) \,|\, t \in \R^+ \}$ where $X$ is called the level and $\phi$ is called the phase. The dynamics are the following: $\{\phi(t)\}$ is a Markov jump process on a state space $\Er$ characterized by its generator $Q$. To each phase $i$ in $\Er$ one associates a rate $c_i \neq 0$. The process $\{X(t)\}$ takes its values in $\R$, has continuous trajectories and is such that
\begin{equation*}
\frac{d}{dt}X(t) = c_{ \phi(t)}.
\end{equation*}
In other words, the level process evolves in a piecewise linear fashion, at rate $c_i$ when $ \phi(t) = i$. It is convenient to reorganize the phase space and partition $\Er$ into two subspaces $\Er_+ =\{ i \in \Er \,|\, c_i > 0 \}$ and $\Er_- = \{ i \in \Er \,|\,  c_i < 0 \}$. Denoting by $C$ the diagonal matrix of rates, we can write $Q$ and $C$ according to this subdivision $\Er = \Er_+ \cup \Er_-$: 
\begin{equation*}
Q =
\begin{bmatrix}
Q_{++} & Q_{+-} \\
Q_{-+} & Q_{--} 
\end{bmatrix} ,~~~~~
C =
\begin{bmatrix}
C_{+} & \\
 & C_{-} 
\end{bmatrix}.
\end{equation*}

To the risk process $\{R(t)\}$ in \eqref{risk_proc_MPH}, we associate the MMFF $\{(X(t),\phi(t))\}$ defined on the infinite phase space $\Er = \Er_+ \cup \Er_-$  where $\Er_+$ and $\Er_-$ are two copies of $\Sr_1 \cup \Sr_2 \cup \cdots$. To differentiate between these sets, we write 
$$\Er_+=\Sr_1^+ \cup \Sr_2^+ \cup \cdots ~~\text{and}~~\Er_-=\Sr_1^- \cup \Sr_2^- \cup \cdots$$
 where $\Sr_k^+$ and $\Sr_k^-$ are two copies of $\Sr_k$. In the sequel, the matrices $Q_{++}$, $Q_{+-}$, $Q_{-+}$ and $Q_{--}$ have the following form according to this state partition:
\begin{equation} \label{Q_MMFF}
\begin{aligned}
& Q_{++} =  \begin{bmatrix}
-\lambda I & 0 & 0 &  \cdots   \\
0& -\lambda I & 0 & \cdots  \\
0&0& -\lambda I &  \cdots  \\
\vdots &\vdots &\vdots & \ddots
\end{bmatrix},~~~~
Q_{+-} =  \begin{bmatrix}
\lambda I & 0 & 0 &  \cdots   \\
0& \lambda I & 0 &  \cdots  \\
0&0& \lambda I&  \cdots  \\
\vdots &\vdots &\vdots  & \ddots 
\end{bmatrix} \\
& Q_{-+} =  \begin{bmatrix}
0 & D_1 & 0 & 0 &  \cdots   \\
0 & 0& D_2 & 0 & \cdots  \\
0 & 0&0& D_3 &  \cdots  \\
\vdots &\vdots &\vdots  &\vdots &
\end{bmatrix},~~~~~~
Q_{--} =  \begin{bmatrix}
A_1 & 0 & 0 &  \cdots   \\
0& A_2 & 0 &  \cdots  \\
0&0& A_3 &  \cdots  \\
\vdots &\vdots &\vdots & \ddots 
\end{bmatrix},
\end{aligned}
\end{equation} 
and $C$ is such that $C_+ = cI_\infty$ and $C_-=-I_\infty$, where $I_\infty$  denotes the identity matrix of infinite dimension.

Built this way and when starting from $X(0)=u$ with initial phase distribution vector $\bm{\pi} = [\bm{\alpha}~~ \bm{0}~~ \bm{0}~~\cdots]$, the process $\{(X(t),\phi(t))\}$  evolves like $\{R(t)\}$ in \eqref{risk_proc_MPH} except that each jump in $\{R(t)\}$ is replaced by a linear decrease of the level $\{X(t)\}$ at the unit rate, for a duration equal to the size of the jump.
Note that the matrices $Q_{++}$ and $Q_{+-}$  need to be constituted by infinitely many blocks in order to maintain the same dependence between the sojourn times in the subspaces $\Sr_k^-$ for the MMFF than for the components of the multivariate phase-type distribution with representation \eqref{gen_G}. Let us describe the first stages of the embedded process in more detail: the MMFF starts from a phase $i \in \Sr_1^+$, chosen according to the vector  $\bm{\alpha}$, and stays in that phase for a period of time Exp($\lambda$) during which the level process increases at rate $c$. Then there is a transition to the corresponding phase in $\Sr_1^-$ and the level starts decreasing at rate $-1$. The duration of this decrease is $\mathrm{PH}(\bm{e}_i^\intercal,A_1)$ where $\bm{e}_i$ is a unitary column vector with $i$-th component equal to one. The matrix $D_1$ contains the absorption rates triggering a transition to $\Sr_2^+$, and the choice of the chosen state in $\Sr_2^+$ at that time also determines the state which will be occupied at the first passage to  $i \in \Sr_2^-$, that is, the initial phase for the phase-type duration representing the second claim.

More formally, the risk process and its associated MMFF are linked by a change of time: denoting 
\[J(t) = \int_0^t \ind{ \phi(s) \in \Er_+} \, ds\]
the time spend by the MMFF in $\Er_+$ up to time $t$ and 
$\mathcal{T}(t)=\inf\{s>0 \,|\, J(s)>t\}$, it holds that
$$\{R(t) \,|\, t\ge 0\}\stackrel{d}{=}\{X(\mathcal{T}(t))\,|\, t\ge 0\}.$$

The levels crossed by $\{R(t)\}$ on a time interval $[0,t[$ are the same as the ones crossed by $\{X(t)\}$ on the interval $[0, \mathcal{T}(t)[$. In particular, the time  of ruin $T$ defined in \eqref{def_Truine} corresponds in the MMFF to the time $J(\tau_0)$ with $\tau_0=\inf\{t>0\,|\, X(t)<0\}$, that is,
$$T\stackrel{d}{=}J(\tau_0).$$ 
Moreover, the variable $-R(T)$ corresponds in the MMFF to the level occupied at the first passage to a phase in $\Er_+$ after $\tau_0$. Finally, the variable $N(T)$ corresponds in the MMFF to the number of transitions from $\Er_+$ to $\Er_-$ up to time $\tau_0$.

\subsection{Ruin probabilities and time of ruin} \label{BBB_tf}

We can now derive an expression for the transform \eqref{def_tf_ruine}, in terms of the blocks of two first passage matrices related to the MMFF $\{(X(t),\phi(t))\}$. First, fix some $\theta\ge 0$. The first matrix $\Psitet$ is such that for all $k,\ell\ge 1$,  $i\in \Er^+_k$, $j\in \Er^-_{\ell}$, 
\begin{equation*}
\left(\Psitet\right)_{(k,i),(\ell,j)}  = \esp{e^{-\theta J(\tau_0)} \,\ind{\phi(\tau_0)=(\ell,j)} \,|\,X(0)=0, \phi(0)=(k,i)}.
\end{equation*}
The second one $\Phitet(x)$ is such that for $x \ge 0$, $i\in \Er^-_k$, $j\in \Er^-_{\ell}$,
\begin{equation*}
\left(\Phitet\right)_{(k,i),(\ell,j)}(x) = \esp{e^{-\theta J(\tau_0)} \,\ind{\phi(\tau_0)=(\ell,j)} \,|\,X(0)=x, \phi(0)=(k,i)}.
\end{equation*}
They give the Laplace transforms of the time $J(\tau_0)$ spent in $\Er_+$ before the first passage to level zero in the MMFF, starting from level zero in an ascending phase (for $\Psitet$) or from level $x$ in a descending phase (for $\Phitet(x)$). 
From the structure \eqref{Q_MMFF} of the generator $Q$ of $\{(X(t),\phi(t))\}$,
they have an upper triangular block structure when written according to the phase subdivision $\Sr_1^+ \cup \Sr_2^+ \cup \cdots$ for $\Er_+$ and $\Sr_1^- \cup \Sr_2^- \cup \cdots$ for $\Er_-$:
\begin{equation}\label{structure_PSI}
\Psitet =  \begin{bmatrix}
\Psitet(1,1) & \Psitet(1,2) & \Psitet(1,3) & \Psitet(1,4) & \cdots  \\
0& \Psitet(2,2) & \Psitet(2,3) & \Psitet(2,4) & \cdots  \\
0&0& \Psitet(3,3) & \Psitet(3,4) & \cdots  \\
0&0&0 & \Psitet(4,4) & \cdots\\
\vdots &\vdots & \vdots  & \vdots  & \ddots
\end{bmatrix},
\end{equation} 
\begin{equation}\label{structure_PHI}
\Phitet(x) =  \begin{bmatrix}
\Phitet(x;1,1) & \Phitet(x;1,2) & \Phitet(x;1,3) & \Phitet(x;1,4) & \cdots  \\
0& \Phitet(x;2,2) & \Phitet(x;2,3) & \Phitet(x;2,4) & \cdots  \\
0&0& \Phitet(x;3,3) & \Phitet(x;3,4) & \cdots  \\
0&0&0 & \Phitet(x;4,4) & \cdots\\
\vdots &\vdots & \vdots  & \vdots  & \ddots
\end{bmatrix}. 
\end{equation} 
The transform \eqref{def_tf_ruine} is easily expressed in terms of the blocks in \eqref{structure_PSI} and \eqref{structure_PHI}:
\begin{proposition} \label{prop_tf}
For any $u\ge 0$, $\theta\ge 0$, $s\ge 1$ and $y\ge 0$,
\begin{equation}\label{eq:mainresult1}
\esp{e^{-\theta T}\ind{T<\infty,\, N(T)\le s,\, -R(T)\ge y}\,|\, R(0)=u} = \sum_{k=1}^{s} \sum_{\ell=0}^{s-k} \bm{\alpha} \Psitet(1,k) \Phitet(u;k,k+\ell) e^{A_{k+\ell}y}\bm{1}.
\end{equation}
\end{proposition}
\begin{proof}
Let $T_u=\inf\{t>0\,|\, R(t)<u\}$ and $\tau_u=\inf\{t>0\,|\, X(t)<u\}$. 
Using the time change given in Section \ref{BBB_MMFF} to switch from the risk process to the MMFF, we see that for fixed $k\ge 1$ and $\ell \ge 0$,
\begin{align*}
&\esp{e^{-\theta T}\ind{T<\infty,\, N(T_u)=k, N(T)=k+\ell,\, -R(T)\ge y}\,|\, R(0)=u}\\
&~ = \esp{e^{-\theta J(\tau_0)}\ind{\tau_0<\infty,\, \phi(\tau_u)\in\Sr_{k}^-,\, \phi(\tau_0)\in\Sr_{k+\ell}^-,\, \eta \ge y} \,|\, X(0)=u, \phi(0)\sim \bm{\pi}} \\
&~  = \esp{\left(e^{-\theta J(\tau_u)}\ind{\tau_u<\infty,\, \phi(\tau_u)\in\Sr_{k}^-} \right) \! \left( e^{-\theta (J(\tau_0)-J(\tau_u))}\ind{\tau_0<\infty,\, \phi(\tau_0)\in\Sr_{k+\ell}^-,\, \eta \ge y} \right) |\, X(0)=u, \phi(0)\sim \bm{\pi}},
\end{align*}
where $\eta$ is the time between $\tau_0$ and the first passage to a phase in $\Er_+$ after $\tau_0$. So, using the strong Markov property,
\begin{align*}
&\esp{e^{-\theta T}\ind{T<\infty,\, N(T_u)=k, \, N(T)=k+\ell,\, -R(T)\ge y}\,|\, R(0)=u}\\
& \quad \quad \quad \quad = \sum_{j\in \Sr_{k}^-}\esp{e^{-\theta J(\tau_u)}\ind{\tau_u<\infty,\, \phi(\tau_u)=j}\,|\, X(0)=u, \phi(0)\sim \bm{\pi}}\\
& \quad \quad \quad \quad \quad \quad \quad \quad \times\, \esp{e^{-\theta J(\tau_0)}\ind{\tau_0<\infty,\, \phi(\tau_0)\in\Sr_{k+\ell}^-,\, \eta \ge y} \, |\, X(0)=u, \phi(0)=j} \\
& \quad \quad \quad \quad = \sum_{j\in \Sr_{k}^-}\left(\bm{\alpha}\Psitet(1,k)\right)_j\left(\Phitet(u;k,k+\ell) e^{A_{k+\ell}y}\bm{1}\right)_j\\
& \quad \quad \quad \quad = \bm{\alpha}\Psitet(1,k)\Phitet(u;k,k+\ell) e^{A_{k+\ell}y}\bm{1}.
\end{align*}
Equation \eqref{eq:mainresult1} follows by summing over $\{k\ge 1, \ell \ge 0\,|\, k+\ell\le s\}$.
\end{proof}

\noindent {\bf Remark.} Taking $\theta=0$ and $y=0$ and  in \eqref{eq:mainresult1}, we obtain 
\begin{equation}
\mathbb{P}\Big(\inf_{s\ge 0} R(s)< 0,\, N(T)\le s \,|\, R(0)=u\Big) = \sum_{k=1}^{s} \sum_{\ell=0}^{s-k} \bm{\alpha} \Psi_0(1,k) \Phi_0(u;k,k+\ell) \bm{1}. \label{partTN}
\end{equation}
The probability of ultimate ruin
\begin{align*}
\mathbb{P}\Big(\inf_{s\ge 0} R(s)< 0\,|\,R(0)=u\Big)
&= \lim_{s \to \infty} \mathbb{P}\Big(\inf_{s\ge 0} R(s)< 0,\, N(T)\le s \,|\, R(0)=u\Big) 
\end{align*}
can be approximated as precisely as desired by computing \eqref{partTN} for $s$ large enough.

\vspace{1em}

In order to apply Proposition \ref{prop_tf} and compute the transform \eqref{eq:mainresult1}, we need a procedure to compute the various blocks of $\Psitet$ and $\Phitet(x)$. To this end, first note that from the definition of $\Phitet(x)$, it is easy to show (see e.g.\ Ramaswami \cite{Ramaswami}) that $\Phitet(x)$ can be expressed under exponential form
\begin{equation*}
\Phitet(x) = e^{\Utet x},
\end{equation*}
where $\Utet$ is a sub-generator with the same block structure as $\Psitet$, i.e.
\begin{equation}\label{structure_U}
\Utet =  \begin{bmatrix}
\Utet(1,1) & \Utet(1,2) & \Utet(1,3) & \Utet(1,4) & \cdots  \\
0& \Utet(2,2) & \Utet(2,3) & \Utet(2,4) & \cdots  \\
0&0& \Utet(3,3) & \Utet(3,4) & \cdots  \\
0&0&0 & \Utet(4,4) & \cdots\\
\vdots &\vdots & \vdots  & \vdots  & \ddots
\end{bmatrix}.
\end{equation}
In the next proposition, we show that the blocks $\Psitet(k,\ell)$ and $\Utet(k,\ell)$ can be obtained recursively. The notation $\delta_{k,\ell}$ is for the Kronecker delta.
\begin{proposition} \label{prop_psi_U}
The matrices $\Psitet(k,\ell)$ are given by
\begin{equation}\label{EQ_PSItet1}
\Psitet(k,k) = \frac{\lambda}{\lambda + \theta} \left(I - \frac{c}{\lambda + \theta}A_k\right)^{-1},
\end{equation}
and, for $\ell >k$,
\begin{equation}\label{EQ_PSItet2}
\Psitet(k,\ell) = \left(\frac{c}{\lambda + \theta} \sum_{v=k}^{\ell-1} \Psitet(k,v) D_v \Psitet(v+1,\ell)\right)\left(I - \frac{c}{\lambda + \theta}A_\ell\right)^{-1}.
\end{equation}
The matrices $\Utet(k,\ell)$ are given by
\begin{equation}
\Utet(k,\ell) = A_k \delta_{k,\ell} + (1-\delta_{k,\ell}) D_k  \Psitet(k+1,\ell) \label{EQ_Utet}
\end{equation}
for $1 \le k \le \ell$.
\end{proposition}
\begin{proof}
We first derive an equation for $(\Psi(k,\ell))_{ij}$ ($i\in \Sr_k^+$, $j\in \Sr_{\ell}^-$). For that, we assume that the MMFF starts from $X(0)=0$ and $\phi(0)=i$. As $[Q_{++}~Q_{+-}] = [-\lambda I~ \lambda I]$, the process $\{\phi(t)\}$ stays in phase $i$ for a duration $\xi \sim \text{Exp}(\lambda)$ before going to the only phase ${i}^- \in \Sr_k^-$ available from $i$. Conditioning on $\xi$, we thus obtain 
\begin{align*}
(\Psitet(k,\ell))_{ij} &= \int_0^{\infty} \lambda e^{-\lambda y}\esp{e^{-\theta J(\tau_0)} \,\ind{\phi(\tau_0)=j} \,|\,X(0)=0, \phi(0)=i, \xi = y}\, dy \\
&= \int_0^{\infty} \lambda e^{-(\lambda + \theta) y}\esp{e^{-\theta J(\tau_0)} \,\ind{\phi(\tau_0)=j} \,|\,X(0)=cy, \phi(0)=i^-}\, dy \\
&= \lambda \int_0^{\infty} e^{-(\lambda + \theta) y} \,(\Phitet(cy;k,\ell))_{ij} \, dy.
\end{align*}
So, in matrix notation,
\begin{equation}
\Psitet(k,\ell) = \frac{\lambda}{c} \int_0^{\infty} e^{-\frac{\lambda+\theta}{c}y} \,\Phitet(y;k,\ell) \, dy, \label{A0}
\end{equation}
and integrating by parts, we obtain
\begin{align*}
\Psitet(k,\ell) & = \bigg[\frac{-\lambda}{\lambda + \theta}e^{-\frac{\lambda + \theta}{c}y} \Phitet(y;k,\ell)\bigg]_{y=0}^\infty + \frac{\lambda}{\lambda + \theta}\int_0^{\infty} e^{-\frac{\lambda + \theta}{c}y} \,\Phitet'(y;k,\ell) \, dy\\
& = \frac{\lambda}{\lambda + \theta}I \delta_{k,\ell} + \frac{\lambda}{\lambda + \theta}\int_0^{\infty} e^{-\frac{\lambda + \theta}{c}y} \,\Phitet'(y;k,\ell) \, dy.
\end{align*}
Since $\Phitet(y)=e^{\Utet y}$, we have that $\Phitet'(y) = \Phitet(y)\Utet$. Using the block structure \eqref{structure_PHI}, \eqref{structure_U}, we find that
\begin{equation*}
\Phitet'(y;k,\ell) = \sum_{v=k}^\ell \Phitet(y;k,v) \Utet(v,\ell),
\end{equation*}
and therefore
\begin{align}
\Psitet(k,\ell) &= \frac{\lambda}{\lambda + \theta} I \delta_{k,\ell} + \frac{\lambda}{\lambda + \theta}\sum_{v=k}^\ell \int_0^{\infty} e^{-\frac{\lambda+\theta}{c}y} \,\Phitet(y;k,v)\, dy \, \Utet(v,\ell) \nonumber \\
&= \frac{\lambda}{\lambda + \theta} I \delta_{k,\ell} + \frac{c}{\lambda + \theta} \sum_{v=k}^\ell  \Psitet(k,v) \Utet(v,\ell), \label{A1}
\end{align}
where the equality \eqref{A1} is obtained by using \eqref{A0} for $\ell = v$.

Let us turn to the matrix $\Phitet(x;k,\ell)$. The MMFF starts now from $X(0)=x$ and with $\phi(0)\in \Sr_k^-$. By conditioning on the time up to the first transition from $\Sr_k^-$ to $\Sr_{k+1}^+$, we obtain
\begin{align*}
\Phitet(x;k,\ell) & = e^{A_kx}\delta_{k,\ell} + (1-\delta_{k,\ell}) \int_0^{x}e^{A_k(x-y)} D_k  \sum_{v=k}^\ell \Psitet(k+1,v) \Phitet(y;v,\ell) \, dy\\
&=e^{A_kx}\Big(\delta_{k,\ell} I + (1-\delta_{k,\ell}) \int_0^{x}e^{-A_ky} D_k  \sum_{v=k}^\ell \Psitet(k+1,v) \Phitet(y;v,\ell) \, dy\Big).
\end{align*}
Differentiating with respect to $x$, we find
\begin{equation*}
\Phitet'(x;k,\ell) = A_k \Phitet(x;k,\ell) + (1-\delta_{k,\ell}) D_k  \sum_{v=k}^\ell \Psitet(k+1,v) \Phitet(x;v,\ell).
\end{equation*}
Finally, we take $x=0$ in the last equality. As $\Phitet(0) = I$ and $\Phitet'(0) = \Utet$, it yields \eqref{EQ_Utet}. Equations  \eqref{EQ_PSItet1} and \eqref{EQ_PSItet2} are obtained by injecting \eqref{EQ_Utet} in \eqref{A1}.
\end{proof}

{Thanks to block subdivision \eqref{structure_PSI}, the dimensions of the matrices that must be repeatedly inverted to numerically compute the matrix P in the above procedure are equal to the dimensions of the blocks $A_{\ell}$.
In most practical cases, these matrices $A_{\ell}$ have moderate size or can be assumed to have a special structure that can be exploited to decrease the computational cost of their inversion. The various blocks of $\Utet$ are easily computed from Equation \eqref{EQ_Utet}. Once these blocks are known, the blocks of the exponential matrix $\Phitet(x)$ can be computed efficiently using the special structure of $\Utet$ by following for instance Kressner \textit{et al.} \cite{Kressner} where the authors propose an efficient incremental procedure to compute the (block-triangular) square matrix constituted by the first $n$ block lines and columns of $e^{\Utet x}$, from the (previously obtained) square matrix constituted by the first $n-1$ block lines and columns of $e^{\Utet x}$.}

\vspace{1em}

\noindent {\bf Particular case where $A_k = A$ and $D_k=D$ for all $k\ge 1$.} In this case, $Y_k \sim \mathrm{PH}(\bm{\gamma}_k,A)$ where $\bm{\gamma}_k = \bm{\alpha} [(-A)^{-1}D]^{k-1}$. In other words, the claim sizes have the same distribution as the inter-arrival times in a Markov arrival process of parameters $(\bm{\alpha},A,D)$ (see Neuts \cite{neuts1979versatile}).  
From Proposition \ref{prop_cov},
\begin{equation*}
Cov(Y_k,Y_\ell) = \bm{\gamma}_k A^{-2} [D(-A)^{-1}]^{\ell-k}\bm{1}- \left(\bm{\gamma}_k (-A)^{-1} \bm{1}\right)\,\left(\bm{\gamma}_\ell (-A)^{-1}\bm{1}\right), \quad \quad \ell \ge k.
\end{equation*}

Observe that the matrices $\Psitet(k,\ell)$ and $\Utet(k,\ell)$ depend here on the difference $h=\ell-k$ only. Denoting $\widehat{\Psi}_{\theta}(h) = \Psitet(k,k+h)$ and $\widehat{U}_{\theta}(h) = \Utet(k,k+h)$ for any $k\ge 0$, we have from Proposition \ref{prop_psi_U} that
\begin{equation}\label{EQ_PSItet1_part}
\widehat{\Psi}_{\theta}(0) = \frac{\lambda}{\lambda + \theta} \left(I - \frac{c}{\lambda + \theta}A\right)^{-1},
\end{equation}
and, for $h >0$,
\begin{equation}\label{EQ_PSItet2_part}
\widehat{\Psi}_{\theta}(h) = \left(\frac{c}{\lambda + \theta} \sum_{v=0}^{h-1} \widehat{\Psi}_{\theta}(v) D \widehat{\Psi}_{\theta}(h \!-\!v \!-\! 1)\right)\left(I - \frac{c}{\lambda + \theta}A\right)^{-1}.
\end{equation}
Moreover,
\begin{equation}
\widehat{U}_{\theta}(h) = A \delta_{h,0} + (1-\delta_{h,0}) D  \widehat{\Psi}_{\theta}(h-1). \label{EQ_Utet_part}
\end{equation}
The formula in Proposition \ref{prop_tf} can be easily adapted with these new matrices. More importantly, the transform \eqref{def_tf_ruine} with $s=\infty$ has a simpler and more compact expression in this particular case:
\begin{corollary} \label{corr1}
When $A_k = A$ and $D_k=D$ for all $k\ge 1$,
\begin{equation}\label{eq:corr1}
\esp{e^{-\theta T}\ind{T<\infty,\, -R(T)\ge y}\,|\, R(0)=u} = \bm{\alpha} \widehat{\Psi}_{\theta} e^{\widehat{U}_{\theta}u} e^{Ay}\bm{1},
\end{equation}
where $\widehat{\Psi}_{\theta}$ is the minimal nonnegative solution of the Riccati equation
\begin{equation} \label{eq_PSIhat}
\frac{\lambda}{c}\, I + Z\left(A - \frac{\lambda + \theta}{c} I \right) + ZDZ = 0,
\end{equation}
and where $\widehat{U}_{\theta} = A + D\widehat{\Psi}_{\theta}$.
\end{corollary}
\begin{proof}
Let us define
\begin{equation}
\widehat{\Psi}_{\theta} = \sum_{h=0}^{\infty} \widehat{\Psi}_{\theta}(h) \quad \text{and}~~ \widehat{U}_{\theta} = \sum_{h=0}^{\infty} \widehat{U}_{\theta}(h).
\end{equation}
Then $\widehat{\Psi}_{\theta}$ is the Laplace transform of the time $J(\tau_0)$ spent in $\Er_+$ before the first passage to level zero in the MMFF, starting from level zero in an ascending phase and $e^{\widehat{U}_{\theta}x}$ is the Laplace transform of the time $J(\tau_0)$ spent in $\Er_+$ before the first passage to level zero in the MMFF, starting from level $x$ in an descending phase. Formula \eqref{eq:corr1} is immediate from these interpretations.

To show that $\widehat{\Psi}_{\theta}$ is a solution of \eqref{eq_PSIhat}, it suffices to sum over $h$ in \eqref{EQ_PSItet1_part} and \eqref{EQ_PSItet2_part}: we have that
\begin{align*}
\widehat{\Psi}_{\theta} 
&= \frac{\lambda}{\lambda + \theta} \left(I - \frac{c}{\lambda + \theta}A\right)^{-1} + \frac{c}{\lambda + \theta} \sum_{h=1}^{\infty} \sum_{v=0}^{h-1} \widehat{\Psi}_{\theta}(v) D \widehat{\Psi}_{\theta}(h \!-\!v \!-\! 1)\left(I - \frac{c}{\lambda + \theta}A\right)^{-1},
\end{align*}
and therefore
\begin{align*}
\widehat{\Psi}_{\theta} \left(I - \frac{c}{\lambda + \theta}A\right)
&= \frac{\lambda}{\lambda + \theta}  + \frac{c}{\lambda + \theta} \sum_{h=0}^{\infty} \sum_{v=0}^{h} \widehat{\Psi}_{\theta}(v) D \widehat{\Psi}_{\theta}(h \!-\!v) \\
&= \frac{\lambda}{\lambda + \theta}  + \frac{c}{\lambda + \theta} \sum_{v=0}^{\infty} \sum_{h=v}^{\infty} \widehat{\Psi}_{\theta}(v) D \widehat{\Psi}_{\theta}(h-v) \\
&= \frac{\lambda}{\lambda + \theta}  + \frac{c}{\lambda + \theta} \widehat{\Psi}_{\theta} D \widehat{\Psi}_{\theta},
\end{align*}
which is equivalent to \eqref{eq_PSIhat}. The fact that \eqref{eq_PSIhat} admit the first passage matrix $\widehat{\Psi}_{\theta}$ as its minimal nonnegative solution is proved e.g.\ in Bean \textit{et al.}\ \cite{Bean}. 
\end{proof}

{The issue of numerically solving the Riccati equations that arise in fluid queues, like Equation \eqref{eq_PSIhat}, has been investigated by various authors over the past several decades. This has led to the development of various stable and efficient procedures such as the logarithmic reduction algorithm (see Latouche and Ramaswami \cite{MAM}), as well as the algorithms proposed in Bini \textit{et al.}\ \cite{Bini} and Guo \cite{GuoAlg}. For instance, the stability of some of the aforementioned algorithms is affirmed in \cite{GuoAlg} when the dimension of the matrix in question is $100\times 100$.}
Note that the matrix $\widehat{\Psi}_{0}$ provides us with a simple criterion to check whether the ruin occurs almost surely in finite time or not:
if $\widehat{\Psi}_{0}$ is stochastic (i.e.\ $\widehat{\Psi}_{0}\bm{1}=\bm{1}$), then the exponential of $\widehat{U}_{0} = A + D\widehat{\Psi}_{0}$ is also stochastic, and therefore we have from \eqref{eq:corr1} that
\begin{equation*}
\widehat{\Psi}_{0}~ \text{is stochastic} \, \Rightarrow \, \mathbb{P}\Big(\inf_{s\ge 0} R(s)< 0 \,|\, R(0)=u\Big) = 1,
\end{equation*}
whatever the value of $u$ is.

\subsection{Risk process in a random environment} \label{A_Random_env}

In this section, we briefly explain how the method presented in Sections \ref{BBB_MMFF} and \ref{BBB_tf} can be easily extended to the analysis of risk processes in a Markov environment. 
Let $\{\xi(t)\,|\, t\ge 0\}$ be a Markov jump process on a state space $\mathcal{X}$ with $|\mathcal{X}| < \infty$, characterized by its generator $\Theta$ and initial probability vector $\bm{q}$.
Consider the process $\{(R(t),\xi(t))\,|\, t\ge 0\}$ in which the reserves has the dynamics
\begin{equation}\label{risk_proc_MPH2}
R(t)= u + \int_0^t c_{\xi(s)} \, ds \, - \sum_{k=1}^{N_{\xi}(t)}Y_k.
\end{equation}
Here, $\{N_{\xi}(t)\}_{t\ge 0}$ is a Markov-modulated Poisson process with intensity $\lambda_i$ when $\xi$ is in state $i \in \mathcal{X}$. As before, for all $n \ge 1$, the vectors $(Y_1, Y_2, \dots, Y_n)$ are independent of $\{N_{\xi}(t)\}$ and follow a multivariate phase-type distribution with representation (\ref{gen_G}). 
In other words, $\{R(t)\}$ is a risk process under the influence of the random environment $\{\xi(t)\}$. The premium rate $c_{\xi(t)}$ and the claim arrival rate $\lambda_{\xi(t)}$ depend on the state of $\xi$ at time $t$.

This risk process can be analysed through an embedded Markov-modulated fluid flow $\{(X(t),\phi(t))\}$ slightly more general than in Section \ref{BBB_MMFF}, with generator $Q$ of the form
\begin{align*} 
& Q_{++} =  \begin{bmatrix}
V_1 & 0 & 0 &  \cdots   \\
0& V_2 & 0 & \cdots  \\
0&0& V_3 &  \cdots  \\
\vdots &\vdots &\vdots & \ddots
\end{bmatrix},~~~~~~~~~~
Q_{+-} =  \begin{bmatrix}
W_1 & 0 & 0 &  \cdots   \\
0& W_2 & 0 &  \cdots  \\
0&0& W_3&  \cdots  \\
\vdots &\vdots &\vdots  & \ddots 
\end{bmatrix} \\
& Q_{-+} =  \begin{bmatrix}
0 & F_1 & 0 & 0 &  \cdots   \\
0 & 0& F_2 & 0 & \cdots  \\
0 & 0&0& F_3 &  \cdots  \\
\vdots &\vdots &\vdots  &\vdots &
\end{bmatrix},~~~~~~
Q_{--} =  \begin{bmatrix}
G_1 & 0 & 0 &  \cdots   \\
0& G_2& 0 &  \cdots  \\
0&0& G_3&  \cdots  \\
\vdots &\vdots &\vdots & \ddots 
\end{bmatrix},
\end{align*}
with, for all $k \ge 1$,
\begin{align*}
& V_k = (\Theta - L) \otimes I_k , ~~~~ W_k = L \otimes I_k, \\
& F_k = I\otimes D_k, ~~~~~~~~~~G_k =  I\otimes A_k,
\end{align*}
where $L = \text{diag}(\lambda_i)_{i \in \mathcal{X}}$, $I_k$ is the identity matrix with the same dimension as $A_k$, $I$ is the identity matrix of dimension $|\mathcal{X}|$ and $\otimes$ is the Kronecker product.
The rate matrices are
\begin{equation*}
C_{+} =  \begin{bmatrix}
C_1 & 0 & 0 &  \cdots   \\
0& C_2 & 0 &  \cdots  \\
0&0& C_3 &  \cdots  \\
\vdots  &\vdots &\vdots & \ddots
\end{bmatrix},~~
C_{-} =  -I_{\infty}.
\end{equation*} 
where $C_k = I_k \otimes \text{diag}(c_i)_{i \in \mathcal{X}}$. The equality in Proposition \ref{prop_tf} is still valid in this case, for matrices $\Psitet$, $\Phitet(x)$ and $\Utet$ with the same definition and block structure as in Section \ref{BBB_tf}. The equations in Proposition \ref{prop_psi_U} need to be adapted to the new generator $Q$, but it can be easily done by following the same argument as in the proof of Proposition \ref{prop_psi_U}.

\vspace{1em}

\noindent {\bf Remark.} This extension also allows us to analyse the same risk process as in \eqref{risk_proc_MPH} except that the inter-arrival times between two claims are no longer exponential but rather distributed as independent $\mathrm{PH}(\bm{\gamma},U)$ random variables. For that, it suffices to use the extension above with $V_k = U \otimes I$ and $W_k = \bm{u} \otimes I$ for all $k$ (where $\bm{u} = -U\bm{1}$), $F_k = D_k \otimes \bm{\gamma}$ and $G_k = A_k$, with the same rate matrices as in Section \ref{BBB_MMFF}.


\subsection{Probability of ultimate ruin}

A natural question about the risk process $\{R(t)\}$ defined by \eqref{risk_proc_MPH} is whether it will get ruined in finite time with probability one, whatever the initial reserves are. One way to answer that question would be to compute the r.h.s.\ of (\ref{partTN}) with $s$ large enough, but is is sometimes more convenient to have simpler criteria. Here, we develop a quick method which can help to assess whether the ultimate ruin is certain or not. It is based on bounding the process (\ref{risk_proc_MPH}) with simple stochastic processes in the stochastic ordering sense.

We first recall the following definitions. First, the \textit{dominating eigenvalue} of a square matrix $A$ is the eigenvalue which has the largest real part. If $A$ is the sub-generator matrix associated to a phase-type distribution, then it can be shown that its dominating eigenvalue is real and strictly negative (see e.g.\ \cite{o1990characterization}).
Next, let $X$ and $Y$ two positive random variables with density functions $f_X$ and $f_Y$. We say that $Y$ \textit{dominates} $X$ in the sense of the usual univariate stochastic order ($X\stochle Y$) if
\[\int_x^\infty f_X(y) \, dy \le \int_x^\infty f_Y(y) \, dy \quad \quad \forall x\ge 0,\]
see e.g.\ Shaked and Shanthikumar \cite{shaked2007stochastic}. 

\begin{lemma}\label{th:stoch_order0}
Let $Y\sim\mathrm{PH}(\bm{\alpha}, A)$. Let $p_0$ and $-\sigma_0$ be the size and the dominant eigenvalue of $A$, and $\nu_0=\max_{1\le i\le p}\{(-A\bm{1})_i\}$. Fix $\nu\in [\nu_0,\infty)$, $p\in \{p_0, p_0+1, p_0+2,\dots\}$ and $\sigma\in (0,\sigma_0]$. Then
\[L\stochle Y\stochle H\]
where $L\sim\mathrm{Exp}(\nu)$ and $H\sim\mathrm{Erlang}(p,\sigma)$.
\end{lemma}
\begin{proof}
Since $\bm{\alpha} e^{Ay}$ is a nonnegative row vector,
\begin{align}
\frac{f_Y(y)}{1-F_Y(y)} & = \frac{\bm{\alpha} e^{Ay}(-A\bm{1})}{\bm{\alpha}e^{Ay}\bm{1}} \le \frac{\bm{\alpha} e^{Ay}(\nu \bm{1})}{\bm{\alpha}e^{Ay}\bm{1}}=\nu \quad\mbox{for all }y\ge 0,\label{eq:hazard1}
\end{align}
In other words, the hazard rate of $Y$ is smaller than the one of $L$, a relation commonly written as $L\hazardle Y$. By \cite[Theorem 1.B.1]{shaked2007stochastic} it implies that $L\stochle Y$. 

To obtain the second inequality, note first that from He \textit{et al.}\ \cite[Corollary 2.1]{he2019moment}, we have $Y\stochle H_0$ where $H_0\sim\mathrm{Erlang}(p_0,\sigma_0)$. Now, if $p\ge p_0$, let $H_1\sim \mathrm{Erlang}(p-p_0,\sigma)$ be independent of $H_0$ (with $H_1=0$ if $p=p_0$). Then $\tfrac{\sigma_0}{\sigma}(H_0+H_1)\ge H_0$ a.s., which yields $H_0 \stochle H$ since $H\sim \tfrac{\sigma_0}{\sigma}(H_0+H_1)$.
\end{proof}

The concept of stochastic ordering is extended to the multivariate setting in the following way: Let $\bm{X}^1$ and $\bm{X}^2$ be two $n$-dimensional random vectors with density functions $f_{\bm{X}^1}$ and $f_{\bm{X}^2}$. We  say that $\bm{X}^2$ \textit{dominates} $\bm{X}^1$ in the sense of the usual stochastic order ($\bm{X}^1\stochle\bm{X}^2$) if 
\[\int_{\bm{x}\in \Gamma} f_{\bm{X}^1}(\bm{x}) d\bm{x}\le \int_{\bm{x}\in \Gamma} f_{\bm{X}^2}(\bm{x}) d\bm{x}\]
for all increasing set $\Gamma\subseteq\mathbb{R}^n$ (i.e.\ for all set $\Gamma$
such that $\bm{x}\in\Gamma$ and $\bm{y}$ is nonnegative implies $\bm{x}+\bm{y}\in\Gamma$).

It is often hard to extend univariate stochastic ordering properties to the multivariate setting. However, an extension of Lemma \ref{th:stoch_order0} to the particular class of multivariate phase-type distributions given by (\ref{joint_density}) can be obtained as follows:
\begin{lemma}\label{th:stoch_order1}
Let $\bm{Y}^{(n)}=(Y_1,\dots,Y_n)$ follow the law given by (\ref{joint_density}). Let $p_n$ and $-\sigma_n$ the largest dimension and dominant eigenvalue amongst the matrices $A_1,A_2,\dots,A_n$, and $\nu_n=\max_{1\le i\le p, 1\le k \le n}\{(-A_k\bm{1})_i\}$. Fix $\nu\in [\nu_n,\infty)$, $p\in \{p_n, p_n+1, p_n+2,\dots\}$ and $\sigma\in (0,\sigma_n]$. Then
\[ \bm{L}^{(n)}\stochle \bm{Y}^{(n)}\stochle \bm{H}^{(n)},\]
where $\bm{L}^{(n)}$ and $\bm{H}^{(n)}$ are $n$-dimensional random vectors with i.i.d.\ entries which are $\mathrm{Exp}(\nu)$- and $\mathrm{Erlang}(p, \sigma)$-distributed, respectively.
\end{lemma}
\begin{proof}
We prove the result for $n=2$ only, the case $n> 2$ easily follows by induction on $n$. For a fixed increasing set $\Gamma\subseteq\mathbb{R}_+^2$, define $\forall y\in\mathbb{R}_+$ the sets $\Gamma_y=\{x\in\mathbb{R} : (y,x)\in\Gamma\}$ and $\Gamma^y=\{x\in\mathbb{R} : (x,y)\in\Gamma\}$. We have that
\begin{align*}
& \int\int_{\Gamma} \bm{\alpha}e^{A_1y_1}D_1e^{A_2y_2}D_2\bm{1} \, dy_1 \, dy_2\\
& \quad \quad \quad \quad \quad  = \int_{0}^\infty (\bm{\alpha}e^{A_1y_1}D_1\bm{1}) \left(\int_{\Gamma_{y_1}} \frac{\bm{\alpha}e^{A_1y_1}D_1}{\bm{\alpha}e^{A_1y_1}D_1\bm{1}}e^{A_2y_2}D_2\bm{1} \, dy_2 \right)dy_1\\
& \quad \quad \quad \quad \quad = \int_{0}^\infty (\bm{\alpha}e^{A_1y_1}D_1\bm{1}) \left(\int_{\Gamma_{y_1}} \frac{\bm{\alpha}e^{A_1y_1}D_1}{\bm{\alpha}e^{A_1y_1}D_1\bm{1}}e^{A_2y_2}(-A_2\bm{1}) \, dy_2 \right)dy_1.
\end{align*}
Note that the quotient in the last two equalities is a probability vector, and thus the function $y_2 \to \tfrac{\bm{\alpha}e^{A_1y_1}D_1}{\bm{\alpha}e^{A_1y_1}D_1\bm{1}}e^{A_2y_2}(-A_2\bm{1})$ is a phase-type density for each fixed $y_1>0$.
We can therefore apply Lemma \ref{th:stoch_order0} to obtain
\[\int_{\Gamma_{y_1}} \frac{\bm{\alpha}e^{A_1y_1}D_1}{\bm{\alpha}e^{A_1y_1}D_1\bm{1}}e^{A_2y_2}(-A_2\bm{1}) \, dy_2\le \int_{\Gamma_{y_1}} f_{p,\sigma} (y_2) \, dy_2,\]
where $f_{p,\sigma}$ is the density function of an Erlang random variable with parameters $p$ and $\sigma$. Consequently,
\begin{align*}
\int\int_{\Gamma} \bm{\alpha}e^{A_1y_1}D_1e^{A_2y_2}D_2\bm{1} \, dy_1 \, dy_2
& \le \int_{0}^\infty (\bm{\alpha}e^{A_1y_1}D_1\bm{1}) \left(\int_{\Gamma_{y_1}} f_{p,\sigma}(y_2) \, dy_2 \right)dy_1\\
& = \int_{0}^\infty \left(\int_{\Gamma^{y_2}}\bm{\alpha}e^{A_1y_1}(-A_1\bm{1}) \, dy_1\right) f_{p,\sigma}(y_2) \, dy_2\\
& \le \int_{0}^\infty \left(\int_{\Gamma^{y_2}}f_{p,\sigma}(y_1) \, dy_1\right) f_{p,\sigma}(y_2) \, dy_2\\
& = \int\int_{\Gamma} f_{p,m}(y_1)  f_{p,m}(y_2) \, dy_1 \, dy_2,
\end{align*}
where we applied Lemma \ref{th:stoch_order0} again in the second-to-last step. This yields $\bm{Y}^{(2)}\stochle \bm{H}^{(2)}$. The second inequality $\bm{L}^{(2)}\stochle \bm{Y}^{(2)}$ follows by similar arguments, replacing $f_{p,\sigma}(t)$ with $\nu e^{-\nu t}$ and inverting the inequalities above. 
\end{proof}

Let us go back to our risk process $\{R(t)\}$ defined by \eqref{risk_proc_MPH}, in which the claim sizes are given by a sequence $\{Y_n\}_{n\ge 1}$ of phase-type random variables with representation (\ref{gen_G}). Lemma \ref{th:stoch_order0} allows us to bound the ultimate ruin probability in this process: assume that
\begin{enumerate}
\item[(A1)] The dimensions of $A_1, A_2,\dots$ are bounded by $p<\infty$,
\item[(A2)] The sequence $\max_i\{(-A_1\bm{1})_i\}, \max_i\{(-A_2\bm{1})_i\},\dots$ is bounded from above by $\nu<\infty$,
\item[(A3)] The dominating eigenvalues of $A_1, A_2,\dots$ are bounded from above by $-\sigma>0$.
\end{enumerate}
Let $\{L_i\}_{i\ge 1}$ and $\{H_i\}_{i\ge 1}$ be two i.i.d.\ sequences with 
$L_i \sim \mathrm{Exp}(\nu)$ and $H_i \sim \mathrm{Erlang}(p, \sigma)$ for all $i>0$, and define the two risk processes $\{R_L(t)\}$ and $\{R_H(t)\}$ such that
\begin{align*}
R_L(t)&= u + ct - \sum_{k=1}^{N(t)}L_k, \quad \quad R_H(t)= u + ct - \sum_{k=1}^{N(t)}H_k,
\end{align*}
where $u$, $c$ and $\{N(t)\}$ are as in \eqref{risk_proc_MPH}. Then,
\begin{proposition}\label{th:boundsruinfinal}
Under the assumptions (A1), (A2) and (A3),
\begin{align*}
\mathbb{P}\Big(\inf_{t\ge 0} R_L(t)< 0\,|\,R_L(0)=u\Big) & \le \mathbb{P}\Big(\inf_{t\ge 0} R(t)< 0\,|\,R(0)=u\Big)\\
& \le \mathbb{P}\Big(\inf_{t\ge 0} R_H(t)< 0\,|\,R_H(0)=u\Big) .
\end{align*}
\end{proposition}
\begin{proof}
Let $\mathcal{N}=\{N(t)\}$ and fix $n\ge 1$, $x_1,\dots,x_n\ge 0$. Define
\begin{align*}
c_\mathcal{N}^n(x_1,\dots,x_n) &= \left\{\begin{array}{cc}1&\mbox{if } u+ cs -\sum_{k=1}^{N(s)\wedge n,} x_k< 0\mbox{ for some }s\ge 0\\0&\mbox{otherwise.}\end{array}
\right.
\end{align*}
It is clear that conditionally on $\mathcal{N}$, the function $c_\mathcal{N}^n:\mathbb{R}^n\rightarrow\mathbb{R}$ is increasing. Thus, by Lemma \ref{th:stoch_order1} and \cite[Section 6.B.1]{shaked2007stochastic},
\[\mathbb{E}(c_\mathcal{N}^n(L_1,\dots,L_n)\,|\,\mathcal{N})\le \mathbb{E}(c_\mathcal{N}^n(Y_1,\dots,Y_n)\,|\,\mathcal{N})\le \mathbb{E}(c_\mathcal{N}^n(H_1,\dots,H_n)\,|\,\mathcal{N}),\]
and therefore
\[\mathbb{E}(c_\mathcal{N}^n(L_1,\dots,L_n))\le \mathbb{E}(c_\mathcal{N}^n(Y_1,\dots,Y_n))\le \mathbb{E}(c_\mathcal{N}^n(H_1,\dots,H_n)).\]
The result follows since
\begin{align*}
\mathbb{E}(c_\mathcal{N}^n(Y_1,\dots,Y_n))&\uparrow \mathbb{P}\Big(\inf_{s\ge 0} R(s)< 0 \,|\,R(0)=u\Big),\\
\mathbb{E}(c_\mathcal{N}^n(L_1,\dots,L_n))&\uparrow \mathbb{P}\Big(\inf_{s\ge 0} R_L(s)< 0\,|\,R_L(0)=u\Big),\\
\mathbb{E}(c_\mathcal{N}^n(H_1,\dots,H_n))&\uparrow \mathbb{P}\Big(\inf_{s\ge 0} R_H(s)< 0\,|\,R_H(0)=u\Big) 
\end{align*}
as $n$ goes to infinity.
\end{proof}

The processes $\{R_L(t)\}$ and $\{R_H(t)\}$ are Cram\'er-Lundberg processes with exponential or Erlang claims, and the corresponding ultimate ruin probabilities are given by very simple and explicit formulae (see e.g.\ Amsussen and Albrecher \cite[Chapter IX]{AsmussenRuin1}). An application of the bounds in Proposition \ref{th:boundsruinfinal} is that they often provide a quick test to check whether our risk process \eqref{risk_proc_MPH} gets almost surely ruined in finite time or not, by using the relations
\[\mathbb{P}\Big(\inf_{s\ge 0} R_L(s)< 0\,|\,R_L(0)=u\Big)=1\Rightarrow \mathbb{P}\Big(\inf_{s\ge 0} R(s)< 0\,|\,R(0)=u\Big)=1,\] and 
\[\mathbb{P}\Big(\inf_{s\ge 0} R_H(s)< 0\,|\,R_H(0)=u\Big)<1\Rightarrow \mathbb{P}\Big(\inf_{s\ge 0} R(s)< 0\,|\,R(0)=u\Big)<1.\]
In most situations where $\prob{\inf_{s\ge 0} R(s)< 0\,|\,R(0)=u}<1$, our bounds are not tight enough to give a good estimate of the ultimate ruin probability, and using formula \eqref{partTN} with $s$ large enough provides much better results. Improving the bounds of Proposition \ref{th:boundsruinfinal} will be the object of further work.



\section{Numerical illustrations} \label{AAA_illustrations}
 
\noindent {\bf Example 2 (continued).} Let us consider the risk process \eqref{risk_proc_MPH} in which the claims are distributed as in Example 2 of Section \ref{BBB_mult_PH}. 
It covers the case where the claims can take two possible forms. For instance, $U\sim \mathrm{PH}(\bm{\beta},B)$ could have a moderate expectation and variance to represent the size of a regular claim. The variable $V\sim \mathrm{PH}(\bm{\gamma},G)$ could have a much higher expectation and/or variance to represent the size of scarce but more severe claims. 
The probabilities $r_k$ and $p_k$ regulate the contagion effect in the kind of claims: for instance, high values for the probabilities $r_k$ mean that a moderate claim is often followed by another moderate claim. High values for the probability $\rho_k$ would mean that once a severe claim occurs, the probability that more severe claims will occur successively is high. The dependence in $k$ of these probabilities can reflect the fact that the company has the means to learn from the past and decrease the risk of severe claim and the number of their successive occurrence.

For this illustration, we assume that the variables $U_k$ are exponentially distributed with parameter $\mu$ and that the $V_i$ are Erlang variables with parameters $m \in \N_0$ and $\mu$. The probabilities $p_k$ are kept constant: $p_k = p$ for all $k$. The probabilities $r_k$ increase with $k$ as $r_k = a + (1-a) (k/(k+1))$ for some $a \in (0,1)$. 

Tables \ref{Table1} and \ref{Table2} show the means, variances and some correlations of the vector $(Y_1,...,Y_8)$ when the parameter values are $\mu = 1$, $m=5$, $r=0.7$,  $p=0.8$ and $a=0.6$. The choice of a sequence $r_k$ increasing to one implies that the expectation and variance of $Y_k$ quickly start to decrease to get closer and closer to the expectation and variance of the exponential distribution with parameter $\mu$. The correlations between $Y_k$ and its direct neighbours $Y_{\ell}$ decrease to zero when $|k-\ell|\to \infty$, but they remain significantly different from zero for quite high values of $|k-\ell|$.

In Figure \ref{Fig1}, we show two Graphs of the ruin probability $\prob{T<\infty, N(T)<s}$ as a function of $s$, for different values of $u$ and $c$ and when $\lambda = 1$, $\mu = 1$, $m=5$, $r=0.7$, $p=0.8$ and $a = 0.6$. Obviously, the ruin probabilities increase with $s$, and we see that in each case $\prob{T<\infty, N(T)<s}$ converges quite quickly as $s$ increases. Taking $s$ between 100 and 200 already yields a result very close to the ultimate ruin probability $\prob{T<\infty}$.
The first graph in Figure \ref{Fig2} shows the same ruin probability but this time as a function of the initial capital $u$ and for different values of $c$, when $s=500$ and the other parameter values are as before.
The second graph in Figure \ref{Fig2} shows the graph of $\prob{T<\infty, N(T)<s}$ as a function of $c$ when $u=0$ and the other parameter values are as before; together with the lower and upper bounds obtained from Proposition \ref{th:boundsruinfinal}. As $s$ is large, we have that $\prob{T<\infty, N(T)<s}$ is approximately equal to the ultimate ruin probability. In this example, we see that the upper bounds is always equal to one and therefore useless. The lower bound, however, starts to decrease as the same time the true ruin probability does, and allows us to say that the ultimate ruin is almost sure for $c$ between 0 and 1.2 (and therefore computing the ruin probability for these values using the exact formula \eqref{partTN} is not needed).
In Figure \ref{Fig3}, we compare the ruin probability $\prob{T<\infty, N(T)<s}$ obtained for two different models: the one with dependent claims as before, and a model where the claims have the same distribution (i.e. $Y_k \sim \mathrm{PH}(\bm{\gamma}_k,A))$ but are independent (see Example 1 in Section \ref{BBB_mult_PH}). Here, $s=500$, $\lambda = 1$, $\mu = 1$, $m=5$, $r=0.7$, $p=0.8$, $a = 0.6$ and $c=1.5$ (left graph) or $c=1.25$ (right graph). It is generally accepted in the literature that more dependence between claims means a higher risk of ruin. This is what we observe here: the model with dependent claims yields higher ruin probabilities when $u$ is large enough, and the ruin probabilities associated to the dependent model remain significantly positive for much larger values of $u$ than the corresponding model with independent claims.
 
\begin{table}[h!!]
\begin{center}
\begin{tabular}{|c|c|c|c|c|c|c|c|c|}
\hhline{~--------}
\multicolumn{1}{c|}{} &  $k=1$ &  $k=2$ & $k=3$ & $k=4$ &  $k=5$ &  $k=6$ & $k=7$ & $k=8$  \\
\hline
$\esp{Y_k}$ & 2.20 & 2.52 & 2.55 & 2.48 & 2.39 & 2.28 & 2.18 & 2.09  \\
 \hline
$\var{Y_k}$ & 5.56  & 6.29  & 6.34 & 6.22 & 6.01 & 5.77 & 5.51 & 5.25  \\
 \hline
\end{tabular}
\end{center}

\vspace{-1em}

\caption{Mean and variance of the first claim sizes in Example 2, when $\mu = 1$, $m=5$, $r=0.7$, $p=0.8$ and $a = 0.6$.}
\label{Table1}
\end{table}

\begin{table}[h!!]
\begin{center}
\begin{tabular}{c|cccccccc}
$Corr(Y_k,Y_{\ell})$ & $\ell=1$ &  $\ell=2$ & $\ell=3$ & $\ell=4$ &  $\ell=5$ &  $\ell=6$ &  $\ell=7$ &  $\ell=8$  \\
 \hline
$k=1$ & 1 & 0.34  & 0.23  & 0.16 &  0.12 &  0.09 & 0.07  & 0.05  \\
$k=2$ & & 1  & 0.40  & 0.28 & 0.21 &  0.15 & 0.12 & 0.09 \\
$k=3$ & & & 1  & 0.42 &  0.31 &  0.23 & 0.18 & 0.14 \\
$k=4$ & & & & 1 &  0.44 &  0.33 & 0.25 & 0.19 \\
$k=5$ & & & & & 1 &  0.45 & 0.34 & 0.26 \\
\end{tabular}
\end{center}

\vspace{-1em}

\caption{Correlation between the first claim sizes in Example 2, when $\mu = 1$, $m=5$, $r=0.7$, $p=0.8$ and $a = 0.6$.}
\label{Table2}
\end{table}

\begin{figure}[h!!]
\begin{center}
\begin{tabular}{c c} 
\includegraphics[scale=0.55]{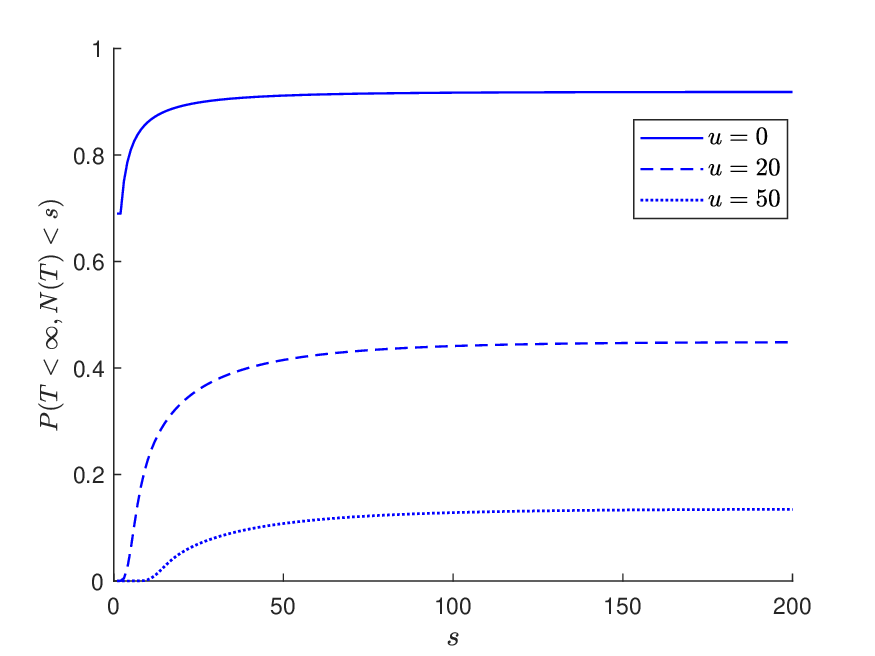} & \hspace{-1.5em} \includegraphics[scale=0.55]{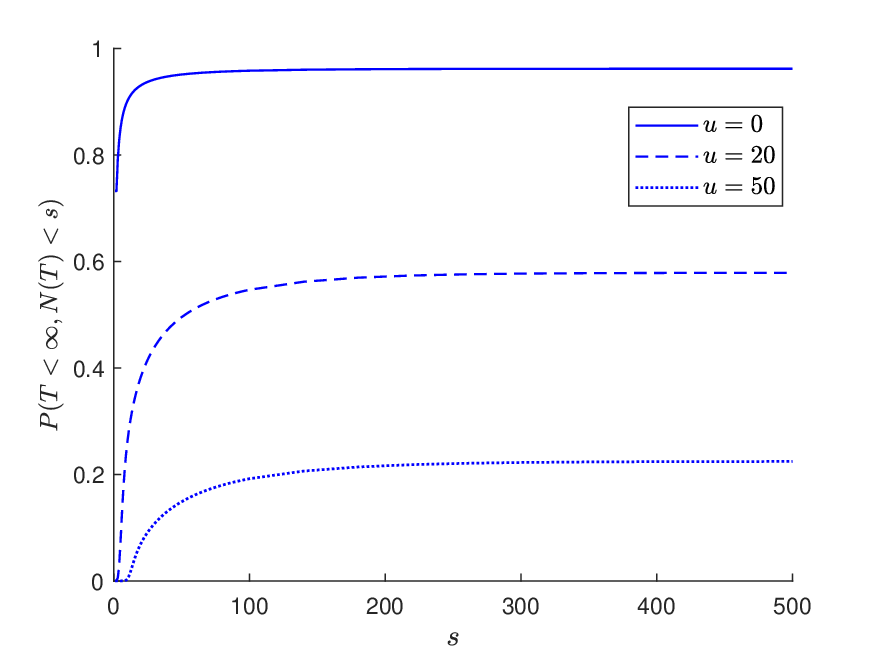}  
\end{tabular}
\end{center}
\vspace{-1em}

\caption{Graphs of the ruin probability $\prob{T<\infty, N(T)<s}$ as a function of $s$ in Example 2, for different values of $u$ when $\lambda = 1$, $\mu = 1$, $m=5$, $r=0.7$, $p=0.8$, $a = 0.6$ and $c=1.5$ (left) or $c=1.25$ (right).} 
\label{Fig1}
\end{figure}

\begin{figure}[h!!]
\begin{center}
\begin{tabular}{c c} 
\includegraphics[scale=0.55]{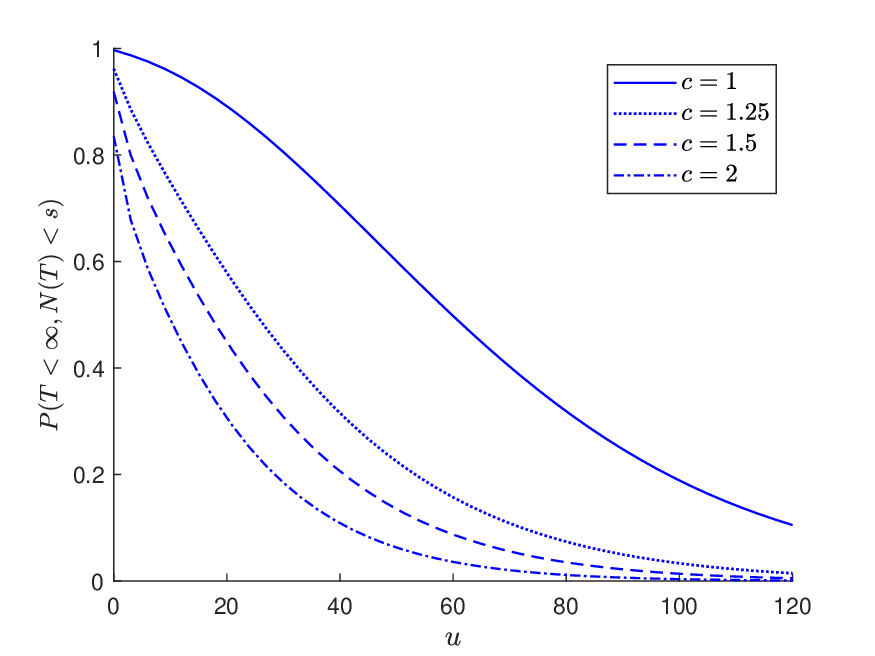} & \hspace{-1.5em} \includegraphics[scale=0.55]{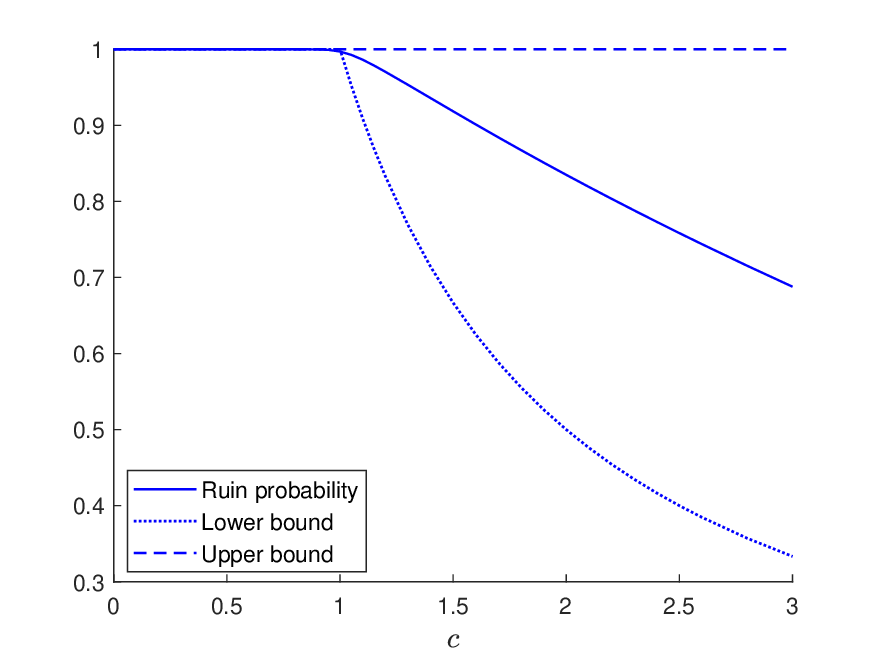}  
\end{tabular}
\end{center}
\vspace{-1em}

\caption{Left: Graphs of the ruin probability $\prob{T<\infty, N(T)<s}$ as a function of the initial capital $u$ in Example 2, for different values of $c$. Right: Comparison between $\prob{T<\infty, N(T)<s}$ and the upper and lower bounds for the ultimate ruin probability when $u=0$. Here $\lambda = 1$, $\mu = 1$, $m=5$, $r=0.7$, $p=0.8$, $a = 0.6$ and $s=500$.} 
\label{Fig2}
\end{figure}

\begin{figure}[h!!]
\begin{center}
\begin{tabular}{c c} 
\includegraphics[scale=0.55]{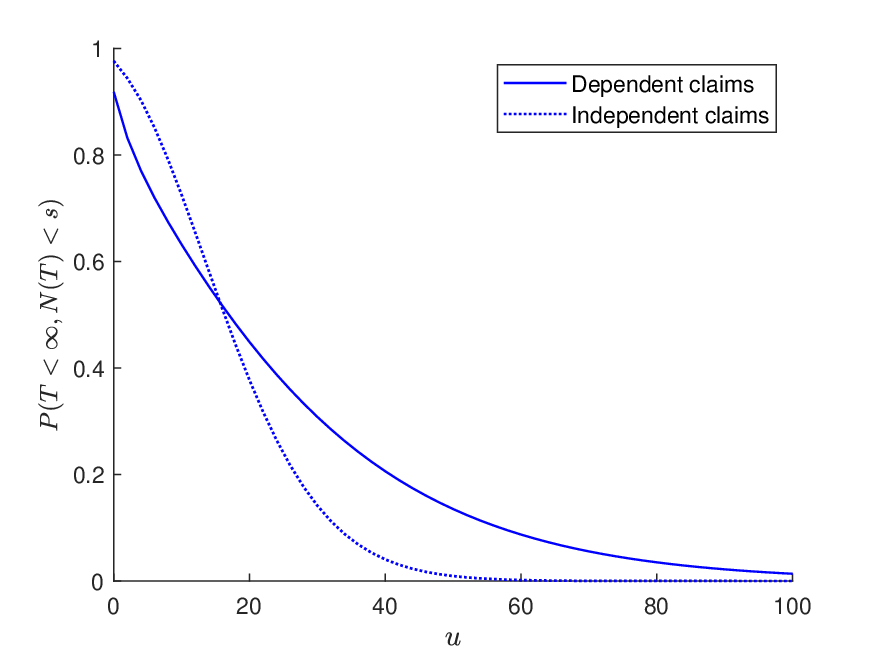} & \hspace{-1.5em} \includegraphics[scale=0.55]{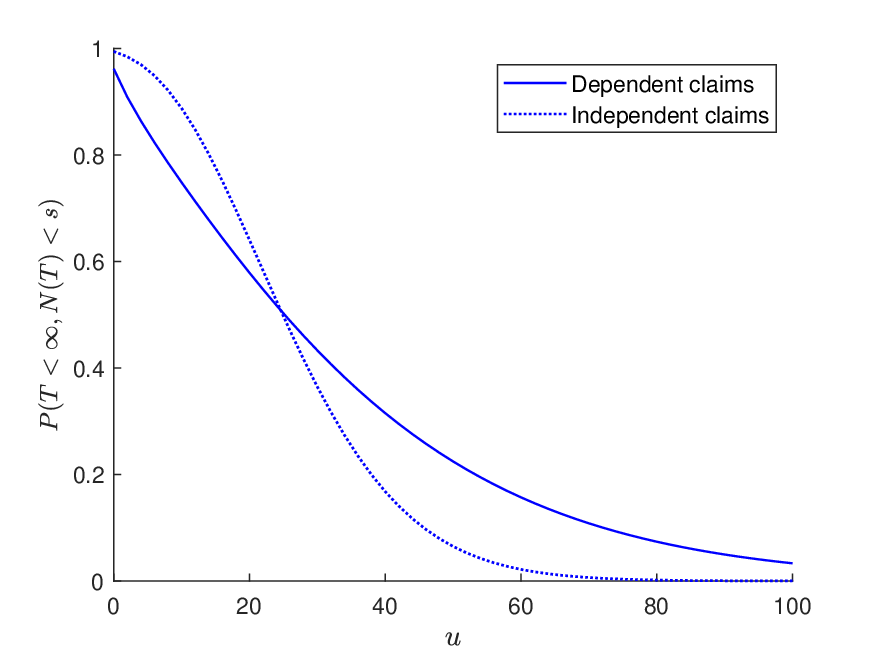}  
\end{tabular}
\end{center}
\vspace{-1em}

\caption{Graphs of the ruin probability $\prob{T<\infty, N(T)<s}$ as a function of $u$ in Example 2, for the model with dependent claims as before and a model where the claims have the same distribution but are independent, and when $s=500$, $\lambda = 1$, $\mu = 1$, $m=5$, $r=0.7$, $p=0.8$, $a = 0.6$ and $c=1.5$ (left) or $c=1.25$ (right).}
\label{Fig3}
\end{figure}

\vspace{1em}

\noindent {\bf Example 3 (continued).} Let us now consider the risk process \eqref{risk_proc_MPH} in which the claims are distributed as in Example 3 of Section \ref{BBB_mult_PH}. 
In the setting of risk processes, it can represent situations where the severity of the first claims (in the sense of the number of exponential stages they go through) has an impact on the severity of the claims to come.

For a fixed integer $m \ge 2$, we choose $P$ as the transition matrix 
$$
P = \begin{bmatrix}
\frac{1}{m-1} & \frac{1}{m-1} & \frac{1}{m-1} & \cdots & \frac{1}{m-1} \\
0 & \frac{1}{m-2} & \frac{1}{m-2} & \cdots & \frac{1}{m-2} \\
0 & 0 & \frac{1}{m-3} & \cdots & \frac{1}{m-3} \\
\vdots & \vdots & \vdots & \ddots & \vdots \\
0 & 0 & 0 & \cdots & 1 \\
\end{bmatrix},
$$
and $\bm{\beta_k} = [1~0~\cdots~0]$ for all $k$. The rates $\mu_k$ and the probabilities $p_k$ will be either constant or given by 
\begin{equation} \label{pk_muk}
\mu_k = 1 + \frac{k}{k+1}, ~~~~ p_k = 0.9 + \frac{k}{20(k+1)},
\end{equation}
so that the mean duration $1/\mu_k$ of each exponential stage decreases over time (from 1 to 0.5) while the probability of going through one more stage slightly increases (from 0.9 to 0.95).

Table \ref{Table3} shows the means, variances and some correlations of $(Y_1,...,Y_8)$ when $m=10$. For larger values of $k$, the expectation and variance of $Y_k$ decrease slowly to stay around 3.35 and 4.6, respectively. 
The correlations between $Y_k$ and $Y_{\ell}$ is significant when $\ell = k+1$ only, but Corr$(Y_{k},Y_{k+1})$ remains around 0.14 as $k$ goes to infinity.

In Figure \ref{Fig4}, we show the ruin probability $\prob{T<\infty, N(T)<s}$ as a function of the initial capital $u$ when $m=10$, $\lambda = 1$ and $p_k$, $\mu_k$ are given in \eqref{pk_muk} and for different values of $s$ and $c$.
In Figure \ref{Fig5}, we change some parameter values: we still take $m=10$ and $\lambda = 1$ but $p_k = 0.95$ and $\mu_k = 2$ are now constant. In that way, the claims are dependent but identically distributed and the model is an example of the particular case discussed at the end of Section \ref{BBB_tf}. The graphs show the ultimate ruin probability as a function of $c$ or $u$. Observe that this probability is equal to one for $c$ lower than 4, whatever the initial capital. This could have been obtained directly by using that $\widehat{\Psi}_{0}$ given in \eqref{eq_PSIhat} is stochastic in these cases.

\begin{table}[h!!]
\begin{center}
\begin{tabular}{|c|c|c|c|c|c|c|c|c|}
\hhline{~--------}
\multicolumn{1}{c|}{} &  $k=1$ &  $k=2$ & $k=3$ & $k=4$ &  $k=5$ &  $k=6$ & $k=7$ & $k=8$  \\
\hline
$\esp{Y_k}$ & 4.81 & 3.34 & 3.57 & 3.46 & 3.46 & 3.44 & 3.43 & 3.42  \\
 \hline
$\var{Y_k}$ & 8.05  & 5.90  & 5.91 & 5.55 & 5.38 & 5.26 & 5.17 & 5.10  \\
 \hline
 $Corr(Y_k,Y_{k+1})$ & 0.23  & 0.10 & 0.12 & 0.12 & 0.13 & 0.13 & 0.13 & 0.13 \\
 \hline
\end{tabular}
\end{center}

\vspace{-1em}

\caption{Mean, variance and correlations of the first claim sizes in Example 3 when $m=10$ and $p_k$, $\mu_k$ are given in \eqref{pk_muk}.}
\label{Table3}
\end{table}

\begin{figure}[h!!]
\begin{center}
\begin{tabular}{c c} 
\includegraphics[scale=0.55]{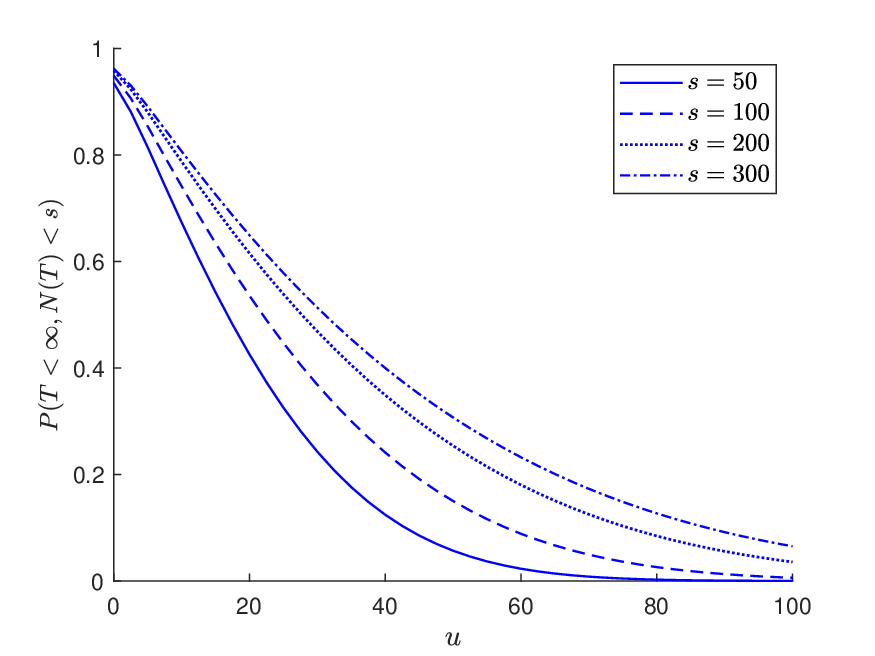} & \hspace{-1.5em} \includegraphics[scale=0.55]{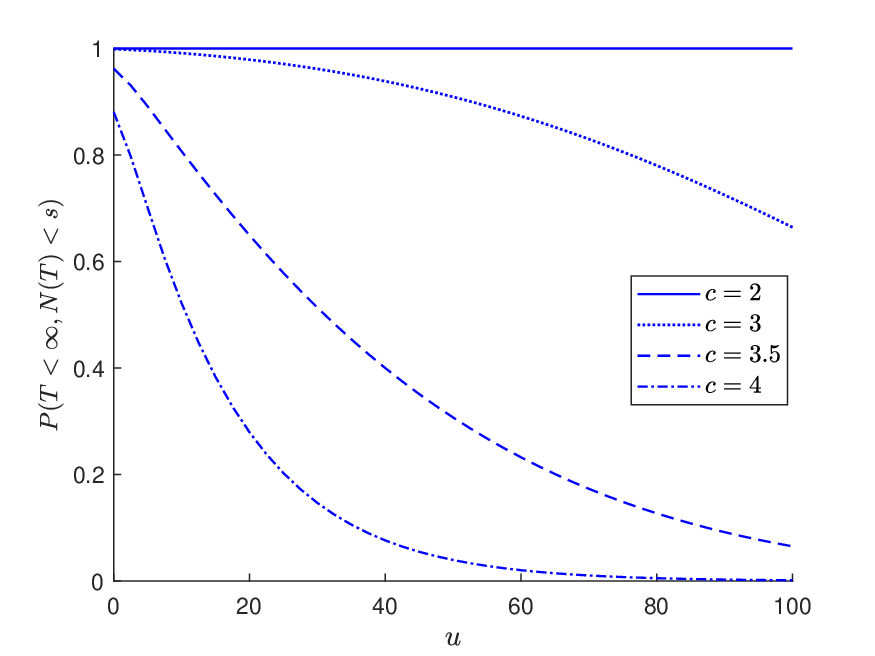}  
\end{tabular}
\end{center}
\vspace{-1em}

\caption{Graphs of the ruin probability $\prob{T<\infty, N(T)<s}$ as a function of the initial capital $u$ in Example 3, for different values of $s$ and $c$ when $m=10$, $\lambda = 1$ and $p_k$, $\mu_k$ are given in \eqref{pk_muk}.  Left graph: $c=3.5$. Right graph: $s=300$.} 
\label{Fig4}
\end{figure}

\begin{figure}[h!!]
\begin{center}
\begin{tabular}{c c} 
\includegraphics[scale=0.55]{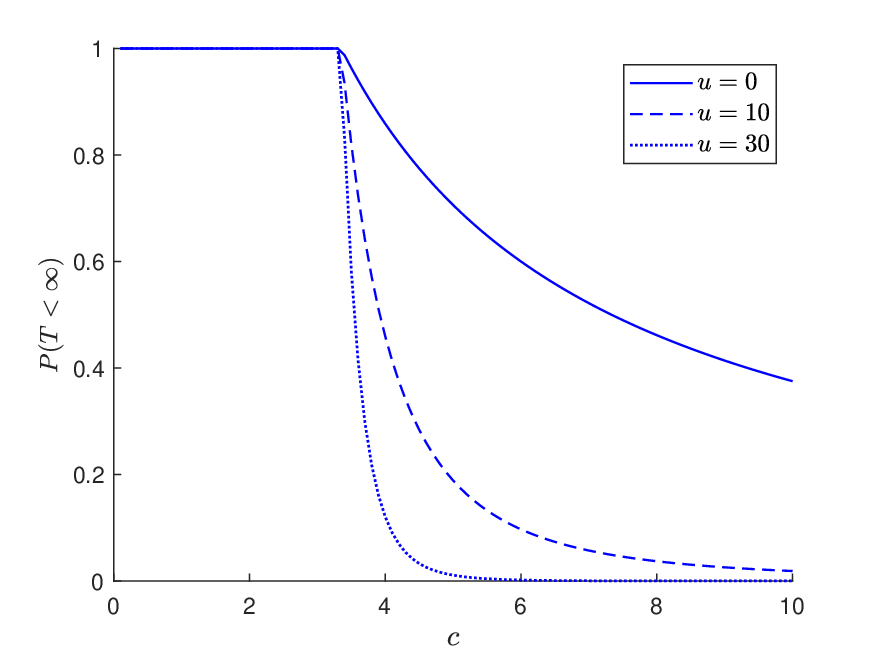} & \hspace{-1.5em} \includegraphics[scale=0.55]{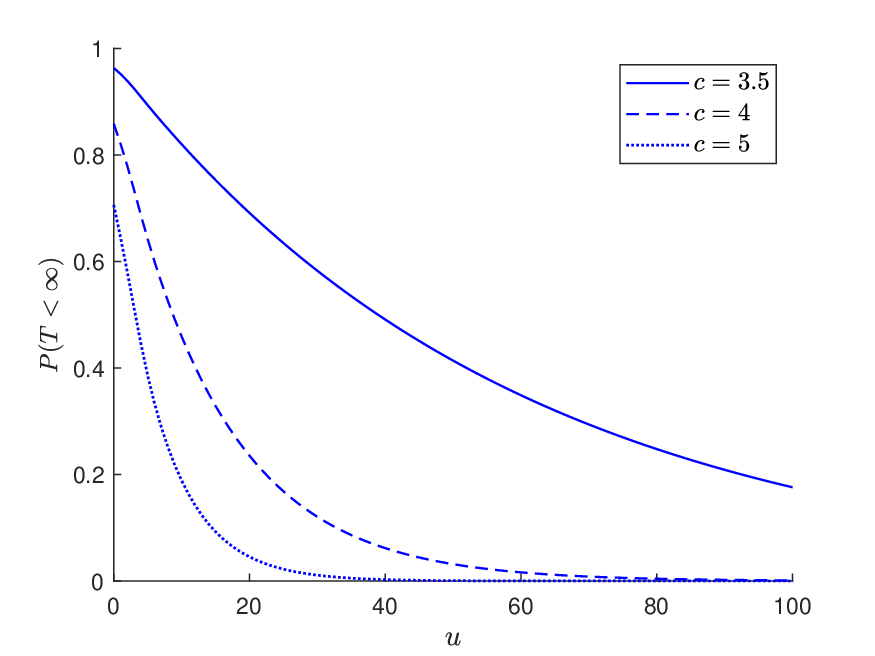}  
\end{tabular}
\end{center}
\vspace{-1em}

\caption{Graphs of the ultimate ruin probability $\prob{T<\infty}$ as a function of the premium $c$ (left) or of the initial reserves $u$ (right) in Example 3, when $m=10$, $\lambda = 1$ and $p_k = 0.95$, $\mu_k = 2$ are constant.} 
\label{Fig5}
\end{figure}
\section*{Acknowledgements}
The authors acknowledge the support of the Australian Research Council Center of Excellence for Mathematical and Statistical Frontiers (ACEMS). Oscar Peralta was additionally supported by the Australian Research Council DP180103106 grant and the Swiss National Science Foundation Project $200021\_191984$.
\section*{Data availability}
Data sharing not applicable to this article as no datasets were generated or analysed during the current study.
\bibliographystyle{abbrv}
\bibliography{Ruin problems with dependent PH claims_revision.bbl} 
\end{document}